\documentclass[12pt]{amsart}%
\usepackage{amsfonts}
\usepackage{amsmath}
\usepackage{amssymb}
\usepackage{amsthm}
\usepackage{mathrsfs}
\usepackage{graphicx}%
\usepackage{amsmath}
\usepackage{setspace}
\usepackage{hyperref}
\usepackage{cite}
\usepackage{amsfonts}
\usepackage{algorithmic}
\usepackage{textcomp}
\usepackage{enumitem}
\usepackage{booktabs,tabularx,array}
\usepackage{cleveref}
\allowdisplaybreaks[4]
\makeatletter
\@addtoreset{equation}{section}
\makeatother

\newtheorem{theorem}{Theorem}[section]

\newtheorem{corollary}[theorem]{Corollary}

\newtheorem{definition}[theorem]{Definition}
\newtheorem{example}[theorem]{Example}

\newtheorem{lemma}[theorem]{Lemma}

\newtheorem{proposition}[theorem]{Proposition}
\newtheorem{remark}[theorem]{Remark}

\newcommand{\R}{\mathbb R}
\newcommand{\C}{\mathbb C}
\newcommand{\Z}{\mathbb Z}
\newcommand{\N}{\mathbb N}
\newcommand{\dd}{\,\mathrm d}
\newcommand{\e}{\mathrm e}
\newcommand{\ind}{\mathbf1}
\newcommand{\norm}[1]{\left\|#1\right\|}

\newcommand{\ip}[2]{\left\langle #1,#2\right\rangle}
\newcommand{\supp}{\operatorname{supp}}
\newcommand{\Ran}{\operatorname{Ran}}
\newcommand{\Cobs}{C_{\mathrm{obs}}}
\newcommand{\Tmin}{T_{\min}}
\newcommand{\QH}{\mathcal Q(H)}
\newcommand{\loc}{\mathrm{loc}}
\newcommand{\step}[1]{\par\medskip\noindent\textit{#1}\quad}
\newcommand{\doi}[1]{\href{https://doi.org/#1}{\nolinkurl{doi:#1}}}
\newcommand{\arxiv}[1]{\href{https://arxiv.org/abs/#1}{\nolinkurl{arXiv:#1}}}
\crefname{theorem}{Theorem}{Theorems}
\crefname{proposition}{Proposition}{Propositions}
\crefname{lemma}{Lemma}{Lemmas}
\crefname{corollary}{Corollary}{Corollaries}

\makeatletter
\def\@makefnmark{}
\makeatother

\begin{document}

\title{Observable sets for the Schr\"odinger equation with potential \(\lvert x\rvert\)}

\author{Lu Chen}
\address[Lu Chen]{Key Laboratory of Algebraic Lie Theory and Analysis of Ministry of Education, School of Mathematics and Statistics, Beijing Institute of Technology, Beijing
100081, PR China}
\email{chenlu5818804@163.com}

\author{Haolin Liu$^{*}$}
\address[Haolin Liu]{Key Laboratory of Algebraic Lie Theory and Analysis of Ministry of Education, School of Mathematics and Statistics, Beijing Institute of Technology, Beijing
100081, PR China}
\email{haolinliu2023@126.com}

\author{Hongyu Liu}
\address[Hongyu Liu]{Department of Mathematics, City University of Hong Kong, Hong Kong SAR, China}
\email{hongyu.liuip@gmail.com, hongyliu@cityu.edu.hk}

\address{}

\thanks{$*$ Corresponding author.}
\thanks{The first author was partly supported by the  National Natural Science Foundation of China (No.12622108), a grant from Beijing Institute of Technology (No.2022CX01002).}
\begin{abstract}
In this paper, we study observable sets for the one-dimensional Schr\"odinger equation associated with the operator \(H=-d^2/dx^2+|x|\) on \(L^2(\R)\). We prove that a measurable set is observable at every positive time if and only if it is thick. For each fixed observation time, the observability constant is uniform over all measurable sets with prescribed thickness parameters. The theorem answers the question concerning the relationship between thick and observable set raised by Gengsheng Wang\cite{Wang26} for this anharmonic oscillator.
\end{abstract}

\maketitle

\section{Introduction}\label{sec:introduction}
\subsection{Problem and background}
We consider the one-dimensional Schr\"odinger equation
\begin{equation}
 \begin{cases}
 i\partial_tu(t,x)=Hu(t,x),&(t,x)\in(0,T)\times\R,\\
 u(0,x)=u_0(x),&u_0\in L^2(\R),
 \end{cases}
 \qquad H=-\dfrac{d^2}{dx^2}+|x|.
 \label{eq:equation}
\end{equation}
We set \(L^2(\R)=L^2(\R;\C)\). The operator \(H\) is the nonnegative self-adjoint operator associated with the closed quadratic form
\begin{equation}
 \mathfrak h[u]=\int_\R\bigl(|u'(x)|^2+|x||u(x)|^2\bigr)\dd x,
 \qquad
 \QH=\left\{u\in H^1(\R):\int_\R|x||u(x)|^2\dd x<\infty\right\}.
 \label{eq:form-intro}
\end{equation}
The solution is \(u(t)=\e^{-itH}u_0\). For a Lebesgue measurable set \(E\subset\R\), we study the observability inequality
\begin{equation}
 \norm{u_0}_2^2\le C\int_0^T\norm{\e^{-itH}u_0}_{L^2(E)}^2\dd t,
 \qquad u_0\in L^2(\R).
 \label{eq:obs}
\end{equation}
The least admissible constant is denoted by \(\Cobs(E,T)\), with the convention that it is \(+\infty\) if no finite constant satisfies \eqref{eq:obs}.

Following Huang--Wang--Wang~\cite{HWW22,HWW26}, we call \(E\) an \emph{observable set at some time} if \(\Cobs(E,T)<\infty\) for at least one \(T>0\), and an \emph{observable set at any time} if this holds for every \(T>0\). The constant is always independent of the initial datum, but may depend on the observation time and the set. Thus, observation at a time \(T\) means observation on the interval \((0,T)\), not at a single instant. We also set
\begin{equation}
 \Tmin(E)=\inf\{T>0:\Cobs(E,T)<\infty\},
 \qquad \inf\varnothing=+\infty.
 \label{eq:Tmin}
\end{equation}
Since the integrand in \eqref{eq:obs} is nonnegative, observability at one time implies observability at every larger time. Firstly we introduce a important definition of observability.

\begin{definition}\label{def:thick}
Let \(L>0\) and \(0<\gamma\le1\). A measurable set \(E\subset\R\) is \emph{\((L,\gamma)\)-thick} if
\begin{equation}
 |E\cap[x,x+L]|\ge\gamma L\qquad(x\in\R).
 \label{eq:thick}
\end{equation}
It is \emph{thick} if \eqref{eq:thick} holds for some such pair. Throughout, \(|\cdot|\) denotes Lebesgue measure.
\end{definition}
At the 2026 ICM, Wang posed open problems concerning the geometric
characterization of observable sets for several evolution equations\cite{{Wang26}}.
A foundational result in this direction is the geometric control
condition (GCC) introduced by Bardos--Lebeau--Rauch in 1992
for the wave equation\cite{{BLR92}}.
This condition relates observability and exact controllability
to the propagation of geometric rays through the observation region
within a prescribed time. For our operator $H$, this open problem is equivalent to determining the necessary and sufficient conditions for a set to be observable at any time or at some time.
We briefly review recent progress on the geometric characterization
of observable sets for Schr\"odinger equations.

The geometry of observable sets can change with the growth of the potential. In~\cite[Theorems~1.1--1.2]{HWW22}, Huang--Wang--Wang proved that thickness characterizes observability at some time for the one-dimensional free equation. For \(-d^2/dx^2+x^{2m}\), with integer \(m\ge1\), their characterization instead uses weak thickness. This means
\[
 \liminf_{r\to\infty}\frac{|E\cap[-r,r]|}{r}>0.
\]
It characterizes observability at some time when \(m=1\), and at any time when \(m\ge2\). In particular, the discrete character of the spectrum alone does not determine the geometric condition or the admissible observation times.

Su--Sun--Yuan~\cite{SSY25} established quantitative observability at every positive time from thick sets for one-dimensional Schr\"odinger equations with real-valued bounded continuous potentials. Their result includes the free equation, but its boundedness assumption does not include \(|x|\). For smooth subquadratic confining potentials, Prouff~\cite{Prouff25} obtained criteria involving the Hamiltonian flow. The sufficient statement used in that comparison concerns a spatial enlargement of the original observation region. It does not directly imply observation on an arbitrary measurable thick set without enlargement. These results illustrate the two distinct issues considered here: allowing an unbounded potential and retaining the original measurable observation set.

For the specific operator in \eqref{eq:equation}, Huang--Wang--Wang~\cite[Theorem~1.3]{HWW26} proved observability at any time, with explicit small-time estimates, under the assumption that \(E^c\) is \(\alpha\)-thin for some \(\alpha>1/2\). The form of this assumption relevant below is
\begin{equation}
 |E^c\cap[x,x+1]|\le M|x|^{-\alpha}
 \qquad(|x|\ge R_0),
 \label{eq:thin}
\end{equation}
for suitable \(M,R_0>0\). They also proved that observability at some time implies weak thickness, and that neither half-line is observable at any finite time; see~\cite[Theorem~1.4]{HWW26}. Their Proposition~4.1 gives an equivalent condition for observability at some time in terms of the uniform positivity of Gram matrices on spectral intervals of fixed width.

Condition \eqref{eq:thin} implies thickness, but the converse is false. For example,
\begin{equation}
 E_{\mathrm{per}}=\bigcup_{k\in\Z}(k,k+1/2)
 \label{eq:periodic-intro}
\end{equation}
is \((1,1/2)\)-thick, whereas \(|E_{\mathrm{per}}^c\cap[k,k+1]|=1/2\) for every integer \(k\). Thus its complement satisfies no positive-power decay condition of the form \eqref{eq:thin}. This example is discussed in~\cite[Remark~3.3]{HWW26}. The question raised in~\cite[Remark~3.8]{HWW26} is whether general thick sets are observable for \eqref{eq:equation}.

For the heat equation on $\mathbb R^n$, Wang, Wang, Zhang and Zhang \cite{WangWangZhangZhang2019} proved that measurable observable sets are precisely the thick sets, and related this characterization to interpolation and spectral inequalities. For Schr\"odinger equations, the propagation mechanism is different; the classical work of Lebeau \cite{Lebeau1992} developed a microlocal framework for controllability. These results provide useful background for the present problem, where the unbounded potential $(V(x)=|x|)$ interacts nontrivially with the geometry of the observation set.

Our aim is to answer this question for every positive observation time and to identify all measurable observable sets for the fixed potential \(|x|\). This is a specific instance of the relation between observation geometry and observation time discussed by Wang~\cite{Wang26}.
\subsection{Main results}
\begin{theorem}\label{thm:main}
Let \(E\subset\R\) be measurable. The following assertions are equivalent:
\begin{enumerate}
 \item \(E\) is thick.
 \item For every \(T>0\), \(\Cobs(E,T)<\infty\).
 \item For some \(T>0\), \(\Cobs(E,T)<\infty\).
\end{enumerate}
Moreover, for every \(T>0\), \(L>0\), and \(0<\gamma\le1\), there exists \(C=C(T,L,\gamma)<\infty\) such that
\begin{equation}
 \norm f_2^2\le C(T,L,\gamma)
       \int_0^T\norm{\e^{-itH}f}_{L^2(E)}^2\dd t,
 \qquad f\in L^2(\R),
 \label{eq:uniform-anytime}
\end{equation}
for every \((L,\gamma)\)-thick set \(E\). If \(E\) is not thick, then for each fixed \(T>0\) there exists a sequence \((f_j)\) in \(L^2(\R)\) such that
\begin{equation}
 \norm{f_j}_2=1,
 \qquad
 \int_0^T\norm{\e^{-itH}f_j}_{L^2(E)}^2\dd t\longrightarrow0.
 \label{eq:main-failure}
\end{equation}
\end{theorem}

The uniform assertion \eqref{eq:uniform-anytime} includes all measurable sets with the same thickness parameters.  The initial data are elements of \(L^2(\R)\). The proof of sufficiency is a contradiction argument and gives existence.

The theorem and \eqref{eq:Tmin} imply
\begin{equation}
 \Tmin(E)=
 \begin{cases}
  0,&E\text{ is thick},\\
  +\infty,&E\text{ is not thick}.
 \end{cases}
 \label{eq:Tmin-dichotomy}
\end{equation}
Thus a finite strictly positive minimal observation time does not occur for this operator. This conclusion differs from the harmonic-oscillator examples in~\cite[Theorem~1.4]{HWW22}, where a positive minimal time occurs for half-space-type observation regions.

We also prove a resolvent estimate that is uniform in the real spectral parameter. As usual in observability theory, this terminology refers to an a priori inequality for \(H-\lambda\), including real parameters in the spectrum.
\begin{proposition}\label{prop:resolvent-intro}
For every \(L>0\) and \(0<\gamma\le1\), there is \(C_0=C_0(L,\gamma)\ge1\) such that every \((L,\gamma)\)-thick set satisfies
\begin{equation}
 \norm u_2^2\le C_0\left(
       \norm{(H-\lambda)u}_2^2+\norm u_{L^2(E)}^2\right),
 \qquad \lambda\in\R,\quad u\in D(H).
 \label{eq:main-resolvent}
\end{equation}
\end{proposition}
\begin{corollary}\label{cor:spectral-intro}
Under the assumptions of \cref{prop:resolvent-intro}, put \(\varepsilon=(2C_0)^{-1/2}\). For every real \(\lambda\),
\begin{equation}
 \norm u_{L^2(E)}^2\ge\frac1{2C_0}\norm u_2^2,
 \qquad
 u\in\Ran\ind_{[\lambda-\varepsilon,\lambda+\varepsilon]}(H).
 \label{eq:cluster-intro}
\end{equation}
\end{corollary}
The proof and the converse spectral implication are given in \cref{sec:consequences}.

The estimate \eqref{eq:main-resolvent} alone gives
\begin{equation}
 \norm f_2^2\le
 \frac{2C_0T}{T^2-\pi^2C_0}
 \int_0^T\norm{\e^{-itH}f}_{L^2(E)}^2\dd t,
 \qquad T>\pi\sqrt{C_0}.
 \label{eq:main-time}
\end{equation}
The time-Fourier argument, including the converse resolvent implication, is given in \cref{sec:time}. The threshold in \eqref{eq:main-time} is only a sufficient time for that argument. It does not yield \eqref{eq:Tmin-dichotomy}. The latter requires the additional compactness argument in \cref{sec:arbitrary-time}.

\subsection{Strategy and comparison with related work}
\subsubsection{A resolvent estimate with a spatial remainder}
The local estimate used below has two parts. When the coefficient in the ordinary differential equation is uniformly bounded, a compactness argument on an interval gives an estimate from every measurable subset of a prescribed positive measure. When its constant part tends to minus infinity, variation of constants reduces the argument to a sine and cosine combination. The latter estimate is closely related to~\cite[Lemmas~4.1--4.2]{SSY25}.

For arbitrary-time observation, a fixed residual coefficient in \eqref{eq:main-resolvent} is insufficient. On the other hand, for the periodic set \eqref{eq:periodic-intro}, the coefficient cannot be replaced in a global estimate by a function tending to zero with the spectral parameter; see \cref{prop:no-decay}. We instead prove
\begin{equation}
 \begin{split}
 \norm u_2^2\le{}&C\norm u_{L^2(E)}^2+\frac CR\norm{(H-\lambda)u}_2^2\\
 &+3\int_{\{\,||x|-\lambda|\le2R\,\}}|u(x)|^2\dd x,
 \qquad R\ge R_*(L,\gamma).
 \end{split}
 \label{eq:turning-intro}
\end{equation}
The constants \(C\) and \(R_*\) depend only on the thickness parameters. The last term is retained throughout the proof. Away from this region, the oscillatory estimate and the quadratic-form identity give the factor \(R^{-1}\) in the residual term.

\subsubsection{Compactness after translation and a temporal Fourier multiplier}
Assume that \eqref{eq:uniform-anytime} fails at a fixed \(T\). Applying \eqref{eq:turning-intro} to the time Fourier transform of \(\chi(t)\e^{-itH}f_j\), where \(\chi\) has compact support in \((0,T)\), shows that a positive amount of its mass lies in \(|\tau+|x||\le2R\). On a unit interval centered at an integer \(m\), multiplication by \(\e^{it|m|}\) places this mass in a fixed bounded interval of temporal frequencies. A weighted selection of the centers and normalization factors then gives translated solutions bounded in a weighted space-time \(L^2\) space.

The translated potentials converge locally to a translate of \(|x|\) minus a constant, or to \(x\) or \(-x\). The solutions have a weak limit, whereas application of the temporal Fourier multiplier gives local strong compactness. A separate estimate of the time tails preserves a nonzero limit. Thickness passes to a weak-star limit of the translated characteristic functions. This implies that the function obtained after the Fourier multiplier vanishes on a positive-measure spatial set. Since its input has compact time support, the identity theorem for entire functions removes the multiplier on this zero set. Unique continuation is then applied to the original weak limit, which satisfies the homogeneous equation.

This argument uses only local regularity of the translated solutions. Their global unweighted \(L^2\) norms need not remain bounded after normalization. In particular, the compact resolvent of \(H\) is not used to infer compactness of the translated sequence.

\subsubsection{Necessity and the precise relation to the open question}
If \(E\) is not thick, it has intervals of increasing length on which the observed measure tends to zero. Unless \(E\) has measure zero, the centers of these intervals tend to infinity along a subsequence. We translate a fixed Gaussian to these centers. After subtracting a constant potential, the translated evolution is compared with the explicit equation having potential \(x\) or \(-x\). The error is estimated in \(L^2\) by Duhamel's formula and Gaussian tails. This constructs actual solutions with vanishing observation on any prescribed finite time interval.

\Cref{thm:main} gives the affirmative answer to~\cite[Remark~3.8]{HWW26} at every positive time and proves the converse geometric implication. In comparison with~\cite[Theorem~1.4(i)]{HWW26}, necessity is strengthened from weak thickness to thickness. In comparison with~\cite[Theorem~1.3]{HWW26}, no decay assumption on the complement remains in the qualitative conclusion.

The method also has a different quantitative scope. Huang--Wang--Wang construct an Ingham-type inequality for spectral clusters, then combine high- and low-frequency estimates to obtain explicit small-time costs. Their half-line counterexample uses the interaction between eigenfunctions and Toeplitz matrices. We do not use those spectral-cluster or Toeplitz estimates in the main proof. Conversely, the present compactness argument does not supply their explicit cost bounds. The distinction between the class of observable sets and the size of the observation constant is therefore maintained.

Resolvent methods for conservative systems are developed in~\cite{Miller05,RTTT05,TW09}. The unique-continuation input in \cref{sec:unique} is a specialization of the open-set theorem of Filippas--Laurent--L\'eautaud~\cite{FLL25}; the reduction from a measurable spatial set is proved in detail. For completeness, \cref{sec:direct-ucp} supplies a direct proof of the needed one-dimensional statement by time regularization and an integral equation. Neither formulation is used as a spectral condition on individual eigenfunctions.

\paragraph{\textbf{Relation to Wang's questions on observation geometry and time.}}
In~\cite[Section~3, questions~(Q1) and~(Q2)(ii)]{Wang26}, Wang asks
how the geometry of observable sets is related to the admissible observation
times and to the minimal observation time, with particular emphasis on
evolution equations in the whole space. For the fixed one-dimensional
operator \(H=-\partial_x^2+|x|\), Theorem~\ref{thm:main} gives a complete
geometric characterization for cylindrical observation regions
\((0,T)\times E\): at each prescribed time \(T>0\), the observable measurable
sets are precisely the thick sets. Consequently,
\[
 T_{\min}(E)=
 \begin{cases}
  0,&\text{if }E\text{ is thick},\\
  +\infty,&\text{if }E\text{ is not thick}.
 \end{cases}
\]
Thus the class of observable spatial sets does not change with the positive
observation time, although the observation constant may depend on that time.
More specifically, the theorem answers the question about thick sets posed
in~\cite[Remark~3.8]{HWW26}, and removes the decay assumption on the complement
from the qualitative arbitrary-time assertion. This provides a complete
answer in the particular cylindrical setting considered here, rather than
a characterization for general time-dependent observation regions, arbitrary
potentials, or higher dimensions. It is also distinct from question~(Q2)(iii)
in~\cite{Wang26}, which concerns the harmonic oscillator.
\subsection{Plan of the paper}
In \cref{sec:preliminaries} we give the operator realization, local regularity facts, and the time-Fourier tools needed for the proof. \Cref{sec:local} proves the local estimates on measurable subsets of an interval. \Cref{sec:global} proves the global resolvent inequalities, including \eqref{eq:turning-intro}. \Cref{sec:unique} establishes unique continuation from a measurable spatial set. The uniform sufficient assertion of \cref{thm:main} is proved in \cref{sec:arbitrary-time}, and its necessary assertion in \cref{sec:necessity}. \Cref{sec:consequences} gives the spectral reformulation and examples. \Cref{sec:time,sec:direct-ucp} contain, respectively, the fixed-resolvent time argument and the direct one-dimensional unique-continuation proof.

\section{The operator and preliminary facts}\label{sec:preliminaries}

\subsection{Conventions}
All function spaces are over \(\C\) unless otherwise specified. We write \(\norm{\cdot}_2\) for the norm in \(L^2(\R)\), and \(\norm{\cdot}_{L^2(I)}\) or \(\norm{\cdot}_{L^2(F)}\) for restrictions to a measurable interval or subset. The inner product is linear in its first argument. Sets that differ by a Lebesgue-null set are identified for purposes of observation and thickness. Constants denoted by \(C\) may be enlarged from line to line; their stated parameter dependence is retained. Constants that enter a subsequent time threshold are given a fixed symbol.

The unitary Fourier transform is
\begin{equation}
 \widehat f(\xi)=(2\pi)^{-1/2}\int_{\R}\e^{-ix\xi}f(x)\dd x.
 \label{eq:fourier}
\end{equation}
We use the same sign convention for the time Fourier transform. Thus \(\mathcal F_t(i\partial_t w)(\tau)=-\tau\widehat w(\tau)\). All finite-time observation operators are bounded because
\begin{equation}
 \int_0^T\norm{\ind_E\e^{-itH}f}_2^2\dd t
 \le T\norm{f}_2^2.
 \label{eq:bounded-observation}
\end{equation}

\subsection{The closed form and its associated operator}
We use the standard closed-form construction of self-adjoint operators; see, for example, Davies~\cite{Davies95}. For clarity, we verify the properties needed in this paper.

\begin{lemma}\label{lem:form}
The space \(\QH\) in \eqref{eq:form-intro}, with norm
\begin{equation}
 \norm{u}_{\QH}^2=\norm{u}_2^2+\norm{u'}_2^2+
                       \norm{|x|^{1/2}u}_2^2,
 \label{eq:form-norm}
\end{equation}
is a Hilbert space, and \(C_c^\infty(\R)\) is dense in it. The form \(\mathfrak h\) is densely defined, closed and nonnegative. Its associated operator satisfies
\begin{equation}
 \begin{split}
 D(H)=\{u\in\QH:
 &\text{there exist }f\in L^2(\R)\text{ such that}\\
 &\mathfrak h(u,v)=\ip f v\text{ for every }v\in\QH\},
 \qquad Hu=f,
 \end{split}
 \label{eq:domain-form}
\end{equation}
where \(\mathfrak h(u,v)=\int u'\overline{v'}+\int |x|u\overline v\).
Moreover,
\begin{equation}
 u\in D(H)\quad\Longrightarrow\quad
 u\in H^2_{\mathrm{loc}}(\R),\qquad Hu=-u''+|x|u\ \text{in distributions}.
 \label{eq:local-reg}
\end{equation}
Conversely, if \(u\in\QH\) and \(-u''+|x|u\in L^2(\R)\) in distributions, then \(u\in D(H)\) and this distribution equals \(Hu\).
\end{lemma}

\begin{proof}
If \((u_j)\) is Cauchy in \eqref{eq:form-norm}, it converges in \(H^1(\R)\) to a function \(u\), while \(|x|^{1/2}u_j\) converges in \(L^2(\R)\) to a function \(w\). By taking subsequences with almost everywhere convergence, \(w=|x|^{1/2}u\) almost everywhere. Thus \(u\in\QH\) and the convergence holds in the form norm. This proves completeness.

To check density, choose \(\zeta\in C_c^\infty(\R)\) equal to one on \([-1,1]\) and set \(\zeta_R(x)=\zeta(x/R)\). Dominated convergence gives
\[
 \norm{(1-\zeta_R)u}_2+
 \norm{(1-\zeta_R)u'}_2+
 \norm{|x|^{1/2}(1-\zeta_R)u}_2\longrightarrow0.
\]
The additional derivative term satisfies \(\norm{\zeta_R'u}_2\le C R^{-1}\norm u_2\to0\). Hence \(\zeta_Ru\to u\) in \(\QH\). For fixed \(R\), convolution with a smooth compactly supported approximate identity approximates \(\zeta_Ru\) in \(H^1\) with support in a fixed compact interval. On that interval the weight \(|x|\) is bounded, so the convergence also holds in the weighted \(L^2\) term. This proves the asserted density.

The closed-form representation theorem gives the nonnegative self-adjoint operator described in \eqref{eq:domain-form}. Testing against \(C_c^\infty(\R)\) yields the distributional identity in \eqref{eq:local-reg}. On any bounded interval, both \(|x|u\) and \(Hu\) belong to \(L^2\), so \(u''\in L^2\) locally. Together with \(u\in H^1\), this gives \(H^2_{\mathrm{loc}}\) regularity. In particular, no independent boundary or transmission condition is imposed at zero: both \(u\) and \(u'\) have the continuous representatives supplied by local \(H^2\) regularity.

Conversely, let \(f=-u''+|x|u\in L^2\) with \(u\in\QH\). Integration by parts against smooth compactly supported functions gives \(\mathfrak h(u,v)=\ip f v\). Both sides are continuous in \(v\) for the form norm. Density extends the identity to every \(v\in\QH\), and \eqref{eq:domain-form} proves the converse.
\end{proof}

\begin{lemma}\label{lem:multipliers}
If \(u\in D(H)\) and \(\chi\in W^{1,\infty}(\R)\) is real-valued, then \(\chi u,\chi^2u\in\QH\). For \(\lambda\in\R\) and \(f=(H-\lambda)u\), one has
\begin{equation}
 \int_{\R}|(\chi u)'|^2\dd x+
 \int_{\R}(|x|-\lambda)\chi^2|u|^2\dd x
 =\operatorname{Re}\int_{\R}f\chi^2\overline u\dd x
   +\int_{\R}|\chi'|^2|u|^2\dd x.
 \label{eq:cutoff-identity}
\end{equation}
\end{lemma}

\begin{proof}
Boundedness of \(\chi\) and \(\chi'\) gives the \(H^1\) and weighted \(L^2\) properties. Therefore \(\chi^2u\) is a permissible test function in \eqref{eq:domain-form}. Taking real parts in that identity, with \(\lambda\ip u{\chi^2u}\) subtracted, gives
\[
 \int\chi^2|u'|^2+2\operatorname{Re}\int\chi\chi'u'\overline u
       +\int(|x|-\lambda)\chi^2|u|^2
 =\operatorname{Re}\int f\chi^2\overline u.
\]
Expansion of \(|(\chi u)'|^2\) gives \eqref{eq:cutoff-identity}. All integrals are finite by the form-domain assumptions. Compact support of \(\chi\) is not required, and no boundary term at infinity is introduced.
\end{proof}

The unitary group \(U(t)=\e^{-itH}\) is strongly continuous on \(L^2(\R)\) and preserves \(D(H)\). For \(u\in D(H)\), the map \(t\mapsto U(t)u\) is continuously differentiable in \(L^2\) and satisfies \(i\partial_tU(t)u=HU(t)u\). These standard consequences of self-adjointness will also be used for unitary conjugates of \(H\).

\begin{lemma}\label{lem:compact}
The embedding \(\QH\hookrightarrow L^2(\R)\) is compact. Consequently \((H+1)^{-1}\) is compact, and \(H\) has an orthonormal basis of eigenfunctions with eigenvalues of finite multiplicity tending to infinity.
\end{lemma}

\begin{proof}
Let \((u_j)\) be bounded in \(\QH\). On each bounded interval it is bounded in \(H^1\), hence has a subsequence converging in \(L^2\) there. A diagonal subsequence converges locally in \(L^2\). Uniformly in \(j\),
\[
 \int_{|x|>R}|u_j|^2\dd x
 \le R^{-1}\int_{\R}|x||u_j|^2\dd x\le C/R.
\]
The local convergence and this tail estimate give convergence in \(L^2(\R)\). The resolvent \((H+1)^{-1}\) maps bounded subsets of \(L^2\) into bounded subsets of \(\QH\), as follows by testing \((H+1)u=f\) against \(u\). Compactness and the spectral theorem finish the proof.
\end{proof}

No asymptotic formula for eigenvalues or eigenfunctions is needed for \cref{thm:main}. Compactness is recorded for the spectral consequences in \cref{sec:consequences}; it is not used to claim that normalized translated solutions are globally compact.

\subsection{Two local facts}
The following interpolation estimate allows a compactness argument on an interval without imposing any boundary condition.

\begin{lemma}\label{lem:interpolation}
For \(L>0\) and \(u\in H^2(0,L)\),
\begin{equation}
 \norm{u'}_{L^2(0,L)}
 \le \frac{\sqrt{12}}{L}\norm u_{L^2(0,L)}
       +(1+\sqrt3)L\norm{u''}_{L^2(0,L)}.
 \label{eq:interval-interpolation}
\end{equation}
In particular, a bound on \(\norm u_2+\norm{u''}_2\) bounds \(u\) in \(H^2(0,L)\).
\end{lemma}

\begin{proof}
Put
\[
 w(x)=\int_0^x(x-s)u''(s)\dd s,
 \qquad p(x)=u(x)-w(x).
\]
Then \(p\) is affine. Young's inequality on the interval, after extension by zero of \(u''\), gives
\[
 \norm w_2\le\frac{L^2}{2}\norm{u''}_2,
 \qquad \norm{w'}_2\le L\norm{u''}_2.
\]
Write \(p(x)=a+b(x-L/2)\). Since
\[
 \norm p_2^2=L|a|^2+\frac{L^3}{12}|b|^2,
 \qquad \norm{p'}_2=\sqrt L\,|b|,
\]
we obtain \(\norm{p'}_2\le\sqrt{12}L^{-1}\norm p_2\). Adding the estimates for \(p'\) and \(w'\) proves \eqref{eq:interval-interpolation}.
\end{proof}

\begin{lemma}\label{lem:ode-unique}
Let \(I\) be a bounded open interval and \(q\in L^\infty(I)\) be real-valued. If \(u\in H^2(I)\) satisfies \(u''=qu\) and vanishes on a set of positive measure, then \(u\equiv0\) on \(I\).
\end{lemma}

\begin{proof}
The function \(u\) has a \(C^1\) representative. Its zero set has an accumulation point \(x_0\) in the interior of \(I\). Continuity and difference quotients along distinct zeroes converging to \(x_0\) give \(u(x_0)=u'(x_0)=0\). The vector \(Y=(u,u')\) is absolutely continuous and satisfies
\[
 Y'(x)=\begin{pmatrix}0&1\\q(x)&0\end{pmatrix}Y(x)
 \quad\text{for almost every }x\in I.
\]
The integral equation, the boundedness of \(q\), and Gr\"onwall's inequality imply \(Y=0\) to the right and to the left of \(x_0\). This argument is valid for complex-valued \(u\) as well.
\end{proof}

\subsection{Local elliptic regularity and time Fourier transforms}
\begin{lemma}\label{lem:local-elliptic}
Let \(J\Subset J_*\) be bounded open intervals. If \(w\in L^2(J_*)\) and \(w''\in L^2(J_*)\) in distributions, then \(w\in H^2(J)\) and
\begin{equation}
 \norm w_{H^2(J)}\le C_{J,J_*}
       \bigl(\norm w_{L^2(J_*)}+\norm{w''}_{L^2(J_*)}\bigr).
 \label{eq:local-elliptic}
\end{equation}
The same estimate can be integrated in an auxiliary variable.
\end{lemma}
\begin{proof}
On a bounded interval \(J_1\) with \(J\Subset J_1\Subset J_*\), choose a point \(x_1\) and set
\[
 z(x)=\int_{x_1}^x(x-y)w''(y)\dd y.
\]
Then \(z\in H^2(J_1)\), and its \(H^2\) norm is bounded by a constant times \(\norm{w''}_{L^2(J_*)}\). The distribution \(u:=w-z\) has second derivative zero on \(J_1\), so it is affine. The coefficients of an affine function on a fixed interval are controlled by its \(L^2\) norm. Since
\begin{equation*}
||u||_{H^2(J_1)}\leq C||u||_{L^2(J_1)}=||w-z||_{L^2(J_1)}\leq ||w||_{L^2(J_1)}+||z||_{L^2(J_1)}.
\end{equation*}
Combining
\begin{equation*}
||w||_{H^2(J_1)}\leq ||u+z||_{H^2(J_1)}\leq ||u||_{H^2(J_1)}+||z||_{H^2(J_1)}
\end{equation*}
then we prove \eqref{eq:local-elliptic}.
\end{proof}

\begin{lemma}\label{lem:graph-fourier}
Let \(f\in D(H)\), \(u(t)=\e^{-itH}f\), and \(\chi\in C_c^\infty(\R)\). Put \(v(t)=\chi(t)u(t)\). Then \(v\in L^2(\R;D(H))\), where \(D(H)\) carries its graph norm, and
\begin{equation}
 \widehat v(\tau)\in D(H)\ \text{for a.e. }\tau,
 \qquad (H+\tau)\widehat v(\tau)=-i\widehat{\chi'u}(\tau).
 \label{eq:graph-fourier}
\end{equation}
Moreover,
\begin{equation}
 \int_\R\norm{(H+\tau)\widehat v(\tau)}_2^2\dd\tau
 =\norm{\chi'}_{L^2(\R)}^2\norm f_2^2.
 \label{eq:graph-fourier-norm}
\end{equation}
\end{lemma}
\begin{proof}
The group preserves \(D(H)\) and \(\norm{Hu(t)}_2=\norm{Hf}_2\), so the graph-norm integrability is immediate. The graph of \(H\) is a closed subspace of \(L^2(\R)\times L^2(\R)\). The Hilbert-valued Fourier transform preserves this closed subspace, as may be seen first for simple graph-valued functions and then by completion. Hence it commutes with \(H\) in this sense. The identity \((i\partial_t-H)v=i\chi'u\), together with our sign convention in \eqref{eq:fourier}, gives \eqref{eq:graph-fourier}. Plancherel's theorem and unitarity give \eqref{eq:graph-fourier-norm}.
\end{proof}

For \(\psi\in C_c^\infty(\R)\), we write \(\psi(D_t)\) for the bounded operator with time-Fourier multiplier \(\psi\). Its convolution kernel \(k_\psi\), with the normalization corresponding to \eqref{eq:fourier}, is a Schwartz function. In particular,
\begin{equation}
 \norm{\psi(D_t)v}_{L^2(\R;X)}\le\norm\psi_\infty\norm v_{L^2(\R;X)},
 \quad
 \norm{\partial_t\psi(D_t)v}_{L^2(\R;X)}
 \le\norm{\tau\psi(\tau)}_\infty\norm v_{L^2(\R;X)}
 \label{eq:time-multiplier}
\end{equation}
for any Hilbert space \(X\). Multiplication by a time-independent bounded spatial function commutes with \(\psi(D_t)\).

\subsection{Two facts about the temporal Fourier multiplier}
We record the local compactness statement used later. It concerns functions after application of a temporal Fourier multiplier; it does not assert spatial regularization for the Schr\"odinger group at a single time.

\begin{lemma}\label{lem:temporal-compactness}
Let \(T>0\), let \(\rho:\R\to(0,\infty)\) be measurable, and assume that both \(\rho\) and \(\rho^{-1}\) are bounded on bounded intervals. Suppose that
\[
 \sup_j\int_0^T\int_\R\rho(y)|U_j(t,y)|^2\dd y\dd t<\infty
\]
and that
\begin{equation}
 (i\partial_t+\partial_y^2-V_j(y))U_j=0
 \quad\text{in }\mathcal D'((0,T)\times\R),
 \label{eq:abstract-local-evolution}
\end{equation}
where the \(V_j\) are independent of time and satisfy
\(\sup_j\norm{V_j}_{L^\infty(J)}<\infty\) for every bounded interval \(J\).
Fix \(\chi\in C_c^\infty(0,T)\) and \(\psi\in C_c^\infty(\R)\), and set
\[
 Z_j=\psi(D_t)(\chi U_j),\qquad F_j=\psi(D_t)(\chi' U_j),
\]
where the arguments of \(\psi(D_t)\) are extended by zero outside \((0,T)\). Then, for all bounded open intervals \(J\Subset J_*\),
\begin{equation}
 \sup_j\left(
 \norm{Z_j}_{L^2(\R_t;H^2(J))}
 +\norm{\partial_t Z_j}_{L^2(\R\times J)}
 +\norm{F_j}_{L^2(\R\times J)}\right)<\infty.
 \label{eq:abstract-local-compact-bounds}
\end{equation}
There is a subsequence converging strongly in \(L^2_{\loc}(\R^2)\). If, along that subsequence,
\[
 U_j\rightharpoonup U
 \quad\text{in }L^2((0,T)\times\R,\rho(y)\dd t\dd y),
\]
then its strong local limit is
\begin{equation}
 Z=\psi(D_t)(\chi U).
 \label{eq:abstract-identified-limit}
\end{equation}
For each bounded \(J\), one also has
\begin{equation}
 \lim_{S\to\infty}\sup_j
   \int_{|t|>S}\int_J|Z_j(t,y)|^2\dd y\dd t=0.
 \label{eq:abstract-time-tail}
\end{equation}
\end{lemma}
\begin{proof}
\step{Step 1. Equation and local bounds.}
For a bounded interval \(J_*\), write
\(c_{J_*}=\operatorname*{ess\,inf}_{J_*}\rho>0\). The assumed weighted bound implies
\[
 \sup_j\norm{U_j}_{L^2((0,T)\times J_*)}^2
 \le c_{J_*}^{-1}\sup_j\norm{U_j}_{L^2(\rho\dd t\dd y)}^2<\infty.
\]
Multiplying \eqref{eq:abstract-local-evolution} by \(\chi\) gives, after extension by zero,
\[
 (i\partial_t+\partial_y^2-V_j)(\chi U_j)=i\chi' U_j
 \quad\text{in }\mathcal D'(\R^2).
\]
There is no endpoint distribution, since \(\chi\) is zero in a neighborhood of both endpoints. The coefficient \(V_j\) is independent of \(t\). Hence time convolution commutes with multiplication by \(V_j\) and with spatial distributional derivatives. We obtain
\begin{equation}
 i\partial_t Z_j=-\partial_y^2 Z_j+V_j Z_j+iF_j
 \quad\text{in }\mathcal D'(\R^2).
 \label{eq:abstract-filtered-equation}
\end{equation}
This identity can alternatively be tested after time Fourier transformation on each bounded spatial interval. Its input and right-hand side have compact time support before the multiplier is applied.

Plancherel's theorem, with values in \(L^2(J_*)\), yields
\begin{align*}
 \norm{Z_j}_{L^2(\R\times J_*)}
 &\le\norm\psi_\infty\norm\chi_\infty
        \norm{U_j}_{L^2((0,T)\times J_*)},\\
 \norm{\partial_tZ_j}_{L^2(\R\times J_*)}
 &\le\norm{\tau\psi(\tau)}_\infty\norm\chi_\infty
        \norm{U_j}_{L^2((0,T)\times J_*)},\\
 \norm{F_j}_{L^2(\R\times J_*)}
 &\le\norm\psi_\infty\norm{\chi'}_\infty
        \norm{U_j}_{L^2((0,T)\times J_*)}.
\end{align*}
Equation \eqref{eq:abstract-filtered-equation} consequently bounds \(\partial_y^2Z_j\) in \(L^2(\R\times J_*)\). By testing first against products of scalar test functions, the equality of the spatial second derivative holds on almost every time slice. Applying \cref{lem:local-elliptic} on \(J\Subset J_*\), squaring, and integrating in time proves \eqref{eq:abstract-local-compact-bounds}.

\step{Step 2. Strong compactness and identification.}
On each bounded rectangle \((-S,S)\times J\), the estimates just obtained control the function and its first derivatives in both variables. The sequence is therefore bounded in \(H^1((-S,S)\times J)\). Rellich compactness gives a strongly convergent subsequence in \(L^2\) on that rectangle. A diagonal extraction over \(S\in\N\) and \(J=(-N,N)\) gives local strong convergence on \(\R^2\).

For every fixed bounded \(J\), the map
\[
 \mathcal A_J:U\longmapsto\left.\psi(D_t)(\chi U)\right|_{\R_t\times J}
\]
is bounded and linear from the weighted Hilbert space in the statement into \(L^2(\R\times J)\); this follows from the first estimate in Step~1. It preserves weak convergence. Thus \(Z_j\rightharpoonup\mathcal A_JU\) in that space. Its local strong limit must agree with this weak limit on every bounded time interval. Exhausting \(\R_t\) and then the spatial line proves \eqref{eq:abstract-identified-limit}.

\step{Step 3. Time tails.}
With the unitary Fourier convention of \eqref{eq:fourier}, define
\[
 k_\psi(t)=\frac1{2\pi}\int_\R\e^{it\tau}\psi(\tau)\dd\tau.
\]
Then \(k_\psi\in\mathcal S(\R)\) and \(\psi(D_t)g=k_\psi*g\). For almost every \(y\),
\[
 |Z_j(t,y)|^2\le
 \left(\int_0^T|k_\psi(t-s)|^2\dd s\right)
 \left(\int_0^T|\chi(s)U_j(s,y)|^2\dd s\right).
\]
Integration over \(J\) and the local bound on \(U_j\) give
\[
 \int_J|Z_j(t,y)|^2\dd y
 \le C_J\int_0^T|k_\psi(t-s)|^2\dd s.
\]
For \(S>T\), a change of variables shows
\begin{equation}
 \int_{|t|>S}\int_J|Z_j|^2\dd y\dd t
 \le C_JT\int_{|r|>S-T}|k_\psi(r)|^2\dd r.
 \label{eq:explicit-temporal-tail}
\end{equation}
The right-hand side tends to zero independently of \(j\), proving \eqref{eq:abstract-time-tail}.
\end{proof}

\begin{lemma}\label{lem:compact-time-injectivity}
Let \(g\in L^2(\R)\) have compact support, and let \(\psi\) be a continuous bounded function that is nonzero on a nonempty open interval. If \(\psi(D_t)g=0\), then \(g=0\).
\end{lemma}
\begin{proof}
Choose \([a,b]\) containing the support of \(g\). By Cauchy--Schwarz, \(g\in L^1\). For \(z\in\C\), define
\[
 G(z)=(2\pi)^{-1/2}\int_a^b\e^{-itz}g(t)\dd t.
\]
For every compact \(K\subset\C\) and integer \(k\ge0\), the functions
\(|t|^k\e^{t\operatorname{Im}z}|g(t)|\), with \(z\in K\), have a common integrable bound. Differentiation under the integral is therefore legitimate at every order, and \(G\) is entire. On the real axis it is the continuous representative of the Fourier transform of \(g\). The hypothesis gives \(\psi(\tau)G(\tau)=0\) almost everywhere. On an open interval where \(\psi\ne0\), continuity implies that \(G\) vanishes everywhere. The identity theorem yields \(G\equiv0\), and injectivity of the \(L^2\) Fourier transform gives \(g=0\).
\end{proof}

\section{A uniform local estimate}\label{sec:local}
We now prove the local estimate used in the global argument. No endpoint condition will be imposed on the functions in this section.

\begin{proposition}\label{prop:local}
Fix \(L>0\), \(0<\gamma\le1\), and \(A_0,M\ge0\). There exists
\(C_{\mathrm{loc}}=C_{\mathrm{loc}}(L,\gamma,A_0,M)\) such that, whenever \(F\subset(0,L)\) is measurable with \(|F|\ge\gamma L\), \(a\in\R\) satisfies \(a\le A_0\), and \(b\in L^\infty(0,L)\) is real-valued with \(\norm b_\infty\le M\), one has
\begin{equation}
 \norm u_{L^2(0,L)}^2
 \le C_{\mathrm{loc}}\left(
     \norm u_{L^2(F)}^2+
     \norm{-u''+(a+b)u}_{L^2(0,L)}^2\right),
 \qquad u\in H^2(0,L).
 \label{eq:local-main}
\end{equation}
The constant is independent of \(a\), \(b\), \(F\) and \(u\) within these classes.
\end{proposition}

The absence of a lower bound for \(a\) is essential. A compactness argument alone treats only a bounded range of coefficients. We first prove that part, and then treat large negative values of \(a\) by an oscillatory calculation.

\subsection{Uniformly bounded coefficients}
\begin{lemma}\label{lem:bounded-local}
Fix \(L>0\), \(0<\gamma\le1\) and \(Q\ge0\). There exists \(C_{\mathrm b}=C_{\mathrm b}(L,\gamma,Q)\) such that
\begin{equation}
 \norm u_{L^2(0,L)}^2
 \le C_{\mathrm b}\left(
    \norm u_{L^2(F)}^2+\norm{-u''+qu}_{L^2(0,L)}^2\right)
 \label{eq:bounded-local}
\end{equation}
for every measurable \(F\subset(0,L)\) with \(|F|\ge\gamma L\), every real-valued \(q\in L^\infty(0,L)\) with \(\norm q_\infty\le Q\), and every \(u\in H^2(0,L)\).
\end{lemma}

\begin{proof}
Suppose the assertion is false. By homogeneity, there are \(q_j,F_j,u_j\) satisfying the stated bounds and
\begin{equation}
 \norm{u_j}_{L^2(0,L)}=1,
 \qquad \norm{u_j}_{L^2(F_j)}^2+
         \norm{f_j}_{L^2(0,L)}^2\longrightarrow0,
 \qquad f_j=-u_j''+q_ju_j.
 \label{eq:bounded-contradiction}
\end{equation}
The differential equation gives
\[
 \norm{u_j''}_2\le Q\norm{u_j}_2+\norm{f_j}_2\le Q+o(1).
\]
By \cref{lem:interpolation}, \((u_j)\) is bounded in \(H^2(0,L)\). In one dimension the embedding into \(C^1([0,L])\) is compact. This also follows directly from the bound
\[
 |u_j'(x)-u_j'(y)|\le\norm{u_j''}_2|x-y|^{1/2}
\]
and the uniform sup-norm bounds obtained from \(H^2\), followed by the Arzel\`a--Ascoli theorem for \(u_j\) and \(u_j'\).

Passing to a subsequence, we may therefore assume
\begin{equation}
 \begin{aligned}
  u_j&\longrightarrow u&&\text{in }C^1([0,L]),\\
  q_j&\rightharpoonup^*q&&\text{in }L^\infty(0,L),\\
  \ind_{F_j}&\rightharpoonup^*\zeta&&\text{in }L^\infty(0,L).
 \end{aligned}
 \label{eq:compact-limits}
\end{equation}
The last two subsequences exist because the relevant bounded subsets of \(L^\infty\) are weak-star sequentially compact when tested against the separable space \(L^1(0,L)\). The limits satisfy
\[
 \norm q_\infty\le Q,\qquad0\le\zeta\le1,
 \qquad \int_0^L\zeta\dd x\ge\gamma L.
\]

For every compactly supported smooth test function, uniform convergence of \(u_j\) and weak-star convergence of \(q_j\) imply \(q_ju_j\to qu\) in distributions. Hence
\begin{equation}
 -u''+qu=0,\qquad \norm u_{L^2(0,L)}=1.
 \label{eq:limit-ode}
\end{equation}
The limit belongs to \(H^2(0,L)\), either by weak compactness in \(H^2\) or by the differential equation.

The varying observation sets cause no loss in the argument. Indeed,
\[
 \left|\int_{F_j}|u_j|^2-\int_0^L\ind_{F_j}|u|^2\right|
 \le L\norm{|u_j|^2-|u|^2}_{L^\infty(0,L)}\longrightarrow0.
\]
Weak-star convergence of \(\ind_{F_j}\) and \eqref{eq:bounded-contradiction} give
\begin{equation}
 \int_0^L\zeta(x)|u(x)|^2\dd x=0.
 \label{eq:limit-observation}
\end{equation}
Thus \(\zeta\) vanishes almost everywhere outside the zero set of \(u\). Since \(\int\zeta\ge\gamma L\) and \(\zeta\le1\), this zero set has positive measure. \Cref{lem:ode-unique} contradicts the normalization in \eqref{eq:limit-ode}. This proves the lemma.
\end{proof}

\subsection{Sine and cosine combinations on measurable subsets}
The following elementary estimate is closely related to the trigonometric estimate in~\cite[Lemma~4.2]{SSY25}. We include the proof both to display a uniform constant and to cover complex coefficients explicitly.

\begin{lemma}\label{lem:trig}
Let \(k>0\) and \(kL\ge1\). For every measurable \(F\subset(0,L)\) with \(|F|\ge\gamma L\) and every \(A,B\in\C\),
\begin{equation}
 \int_F|A\cos(kx)+B\sin(kx)|^2\dd x
 \ge\frac{\gamma^3L}{800}(|A|^2+|B|^2).
 \label{eq:trig-lower}
\end{equation}
Consequently, for \(v(x)=A\cos(kx)+B\sin(kx)\),
\begin{equation}
 \norm v_{L^2(0,L)}\le D_\gamma\norm v_{L^2(F)},
 \qquad D_\gamma=\sqrt{800}\,\gamma^{-3/2}.
 \label{eq:trig-observation}
\end{equation}
\end{lemma}

\begin{proof}
Fix \(\theta\in\R\) and \(0<\epsilon\le1/2\). On each interval of length \(\pi\) in the variable \(z\), the set where \(|\cos z|<\epsilon\) has length \(2\arcsin\epsilon\le4\epsilon\). An interval of length \(kL\) meets at most \(kL/\pi+2\) such period intervals. Changing variables \(z=kx-\theta\) therefore gives
\begin{equation}
 \left|\{x\in(0,L):|\cos(kx-\theta)|<\epsilon\}\right|
 \le 4\epsilon\left(\frac L\pi+\frac2k\right)
 \le10\epsilon L.
 \label{eq:trig-small}
\end{equation}
In the last inequality we used \(kL\ge1\) and \(4/\pi+8<10\).

Take \(\epsilon=\gamma/20\). Since \(|F|\ge\gamma L\), a subset of \(F\) of measure at least \(\gamma L/2\) lies outside the set in \eqref{eq:trig-small}. Thus
\begin{equation}
 \int_F\cos^2(kx-\theta)\dd x
 \ge \frac{\gamma L}{2}\left(\frac\gamma{20}\right)^2
 =\frac{\gamma^3L}{800}.
 \label{eq:cos-lower}
\end{equation}
For real \(A,B\) not both zero, write \(A\cos(kx)+B\sin(kx)=\sqrt{A^2+B^2}\cos(kx-\theta)\) and use \eqref{eq:cos-lower}. The assertion is trivial when both vanish. For complex \(A,B\), apply the real result separately to the real and imaginary parts and add. Finally, pointwise Cauchy--Schwarz yields \(|v(x)|^2\le|A|^2+|B|^2\), so \(\norm v_{L^2(0,L)}^2\le L(|A|^2+|B|^2)\). This proves \eqref{eq:trig-observation}.
\end{proof}

\subsection{Large negative constant part}
\begin{lemma}\label{lem:large-negative}
With \(D_\gamma\) as in \cref{lem:trig}, define
\begin{equation}
 k_*=
 \max\{1,L^{-1},2(D_\gamma+1)LM\}.
 \label{eq:kstar}
\end{equation}
If \(k\ge k_*\), \(\norm b_\infty\le M\), and \(f=-u''+(-k^2+b)u\) on \((0,L)\), then
\begin{equation}
 \norm u_{L^2(0,L)}
 \le2D_\gamma\norm u_{L^2(F)}
       +\frac{2(D_\gamma+1)L}{k}\norm f_{L^2(0,L)}.
 \label{eq:negative-local}
\end{equation}
\end{lemma}

\begin{proof}
The equation can be written as
\begin{equation}
 u''+k^2u=bu-f.
 \label{eq:variation-equation}
\end{equation}
Let
\[
 v(x)=u(0)\cos(kx)+\frac{u'(0)}k\sin(kx).
\]
The traces are well defined for \(u\in H^2(0,L)\). The difference \(u-v\) has zero Cauchy data at zero. The variation-of-constants formula, valid for an \(L^2\) forcing term, is
\begin{equation}
 u(x)-v(x)=\frac1k\int_0^x\sin(k(x-y))(b(y)u(y)-f(y))\dd y.
 \label{eq:variation-formula}
\end{equation}
For completeness, differentiating the integral in the weak sense gives \eqref{eq:variation-equation} and zero Cauchy data; uniqueness for the constant-coefficient initial-value problem proves the identity.

Young's inequality for the kernel \(\ind_{(0,L)}(z)\sin(kz)/k\) gives
\begin{equation}
 \norm{u-v}_{L^2(0,L)}
 \le\frac Lk\bigl(M\norm u_{L^2(0,L)}+\norm f_{L^2(0,L)}\bigr).
 \label{eq:variation-bound}
\end{equation}
Using \cref{lem:trig} and \(\norm{u-v}_{L^2(F)}\le\norm{u-v}_{L^2(0,L)}\), we obtain
\begin{align}
 \norm u_{L^2(0,L)}
 &\le\norm v_{L^2(0,L)}+\norm{u-v}_{L^2(0,L)}\notag\\
 &\le D_\gamma\norm u_{L^2(F)}
        +(D_\gamma+1)\norm{u-v}_{L^2(0,L)}\notag\\
 &\le D_\gamma\norm u_{L^2(F)}
       +\frac{(D_\gamma+1)LM}{k}\norm u_{L^2(0,L)}
       +\frac{(D_\gamma+1)L}{k}\norm f_{L^2(0,L)}.
 \label{eq:absorb-local}
\end{align}
By \eqref{eq:kstar}, the coefficient of \(\norm u_{L^2(0,L)}\) on the right is at most \(1/2\). Absorbing that term proves \eqref{eq:negative-local}.
\end{proof}

\subsection{Completion of the local estimate}
\begin{proof}[Proof of \cref{prop:local}]
For \(a\le-k_*^2\), take \(k=\sqrt{-a}\) and apply \cref{lem:large-negative}. Since \(k\ge1\), squaring \eqref{eq:negative-local} gives
\[
 \norm u_{L^2(0,L)}^2
 \le8D_\gamma^2\norm u_{L^2(F)}^2
      +8(D_\gamma+1)^2L^2\norm{-u''+(a+b)u}_{L^2(0,L)}^2.
\]
For \(-k_*^2\le a\le A_0\), the full coefficient has the uniform bound
\[
 \norm{a+b}_\infty\le Q_*:=\max\{k_*^2,A_0\}+M.
\]
\Cref{lem:bounded-local} applies with \(Q=Q_*\). Taking the maximum of its constant and the two constants in the preceding display gives \(C_{\mathrm{loc}}(L,\gamma,A_0,M)\) in \eqref{eq:local-main}. This proves the proposition for every \(a\le A_0\).
\end{proof}

\begin{remark}\label{rem:local-quantifiers}
The observation set and the coefficient may both vary in \cref{prop:local}. The result is not obtained by fixing \(F\) before taking a spectral limit. Also, no Dirichlet, Neumann or periodic condition is imposed at the endpoints. These facts allow the estimate to be applied to restrictions of a global function on each interval of a partition.
\end{remark}

\section{Global resolvent estimates}\label{sec:global}
\subsection{A uniform estimate with a fixed residual coefficient}
We first prove \cref{prop:resolvent-intro}. The proof separates the region where \(|x|-\lambda\) is bounded above from the region where it is uniformly positive.

\begin{proof}[Proof of \cref{prop:resolvent-intro}]
Fix \(\lambda\in\R\) and \(u\in D(H)\), and put
\begin{equation}
 f=(H-\lambda)u,\qquad q_\lambda(x)=|x|-\lambda.
 \label{eq:q-global}
\end{equation}
By \cref{lem:form}, the equation \(-u''+q_\lambda u=f\) is valid locally in \(L^2\), and \(u\) belongs to \(H^2\) on every bounded interval.

\step{Step 1. Intervals on which the local estimate applies.}
Partition \(\R\), up to endpoints, into
\[
 I_j=(jL,(j+1)L),\qquad j\in\Z.
\]
Let
\begin{equation}
 \mathcal J_\lambda=
 \left\{j\in\Z:\min_{\overline I_j}q_\lambda\le2\right\},
 \qquad \mathcal G_\lambda=\bigcup_{j\in\mathcal J_\lambda}I_j.
 \label{eq:selected-intervals}
\end{equation}
For a selected interval, write
\[
 q_\lambda(x)=a_j+b_j(x),\qquad
 a_j=\min_{\overline I_j}q_\lambda.
\]
The function \(x\mapsto|x|\) is \(1\)-Lipschitz, so
\begin{equation}
 a_j\le2,\qquad 0\le b_j(x)\le L\quad(x\in I_j).
 \label{eq:oscillation-potential}
\end{equation}
Since \(E\) is \((L,\gamma)\)-thick, \(|E\cap I_j|\ge\gamma L\). Translate the interval to \((0,L)\) and apply \cref{prop:local} with \(A_0=2\) and \(M=L\). With
\begin{equation}
 C_*=\max\{1,C_{\mathrm{loc}}(L,\gamma,2,L)\},
 \label{eq:Cstar}
\end{equation}
we obtain
\[
 \int_{I_j}|u|^2\dd x
 \le C_*\left(\int_{E\cap I_j}|u|^2\dd x+\int_{I_j}|f|^2\dd x\right).
\]
Summing over the disjoint selected intervals gives
\begin{equation}
 \int_{\mathcal G_\lambda}|u|^2\dd x
 \le C_*\bigl(\norm u_{L^2(E)}^2+\norm f_2^2\bigr).
 \label{eq:global-good}
\end{equation}
The set \(\{q_\lambda\le2\}\) is contained in \(\mathcal G_\lambda\) except at partition endpoints. Thus, on putting
\[
 A_\lambda=\int_{\{q_\lambda\le2\}}|u|^2\dd x,
\]
we have
\begin{equation}
 A_\lambda\le C_*\bigl(\norm u_{L^2(E)}^2+\norm f_2^2\bigr).
 \label{eq:Alambda}
\end{equation}
For a fixed \(\lambda\), the selected set may be empty. Then \eqref{eq:Alambda} is still valid, with \(A_\lambda=0\). No constant depends on how many intervals are selected.

\step{Step 2. The region where the coefficient is positive.}
Define the Lipschitz cutoff
\begin{equation}
 \chi_\lambda(x)=
 \begin{cases}
  0,&q_\lambda(x)\le1,\\
  q_\lambda(x)-1,&1<q_\lambda(x)<2,\\
  1,&q_\lambda(x)\ge2.
 \end{cases}
 \label{eq:chi-global}
\end{equation}
Then \(0\le\chi_\lambda\le1\), \(|\chi_\lambda'|\le1\) almost everywhere, and the derivative vanishes outside \(\{1<q_\lambda<2\}\). On the support of \(\chi_\lambda\), the coefficient \(q_\lambda\) is at least one. \Cref{lem:multipliers} gives
\begin{equation}
 \norm{\chi_\lambda u}_2^2
 \le \norm f_2\norm{\chi_\lambda u}_2
      +\int_{\{1<q_\lambda<2\}}|u|^2\dd x.
 \label{eq:positive-region}
\end{equation}
Here \(\norm{\chi_\lambda^2u}_2\le\norm{\chi_\lambda u}_2\) was used for the forcing term. Young's inequality yields
\begin{equation}
 \norm{\chi_\lambda u}_2^2
 \le\norm f_2^2+2A_\lambda.
 \label{eq:chi-mass}
\end{equation}

\step{Step 3. Combination.}
Since \(\chi_\lambda=1\) on \(\{q_\lambda\ge2\}\),
\begin{align*}
 \norm u_2^2
 &\le A_\lambda+\norm{\chi_\lambda u}_2^2\\
 &\le3A_\lambda+\norm f_2^2\\
 &\le3C_*\norm u_{L^2(E)}^2+(3C_*+1)\norm f_2^2.
\end{align*}
Taking
\begin{equation}
 C_0=3C_*+1
 \label{eq:C0-definition}
\end{equation}
proves \eqref{eq:main-resolvent} for all real \(\lambda\). Every constant depends only on \(L\) and \(\gamma\), as asserted.
\end{proof}

\begin{remark}
The cutoff in \eqref{eq:chi-global} is not compactly supported, but it is used only as a bounded Lipschitz multiplier in the closed-form identity. Thus the preceding proof does not assume unproved decay of \(u'\) at infinity. The same argument applies for negative spectral parameters, including those for which the set in \eqref{eq:selected-intervals} is empty.
\end{remark}

\subsection{An estimate with a remainder near the turning region}
The following form of the resolvent estimate is the one used in the arbitrary-time proof. Its remainder is a spatial term depending on the spectral parameter; no compactness of this remainder is asserted.

\begin{proposition}\label{prop:turning}
There are \(R_*=R_*(L,\gamma)\ge1\) and \(C_1=C_1(L,\gamma)\) such that, for every \((L,\gamma)\)-thick set \(E\), every \(R\ge R_*\), \(\lambda\in\R\), and \(u\in D(H)\),
\begin{equation}
 \begin{split}
 \norm u_2^2\le {}&C_1\norm u_{L^2(E)}^2
 +\frac{C_1}{R}\norm{(H-\lambda)u}_2^2\\
 &+3\int_{\{x:\,\lvert |x|-\lambda\rvert\le2R\}}|u(x)|^2\dd x.
 \end{split}
 \label{eq:turning-resolvent}
\end{equation}
\end{proposition}
\begin{proof}
Write \(q_\lambda=|x|-\lambda\) and \(g=(H-\lambda)u\). Partition the line into the intervals \(I_k=(kL,(k+1)L)\).

\step{Step 1. The region \(q_\lambda\le-2R\).}
Select every interval whose closure meets \(\{q_\lambda\le-2R\}\). On each selected interval let
\[
 k_I^2=-\min_{\overline I}q_\lambda\ge2R,
 \qquad b_I=q_\lambda+k_I^2.
\]
Then \(0\le b_I\le L\), because \(|x|\) is \(1\)-Lipschitz. Choose \(R_*\) large enough that \(\sqrt{2R_*}\) exceeds the threshold in \cref{lem:large-negative} with interval length \(L\), relative measure \(\gamma\), and \(M=L\). That lemma gives
\[
 \norm u_{L^2(I)}\le
 2D_\gamma\norm u_{L^2(E\cap I)}
 +\frac{2(D_\gamma+1)L}{k_I}\norm g_{L^2(I)}.
\]
After squaring and summing over the disjoint selected intervals,
\begin{equation}
 \int_{\{q_\lambda\le-2R\}}|u|^2\dd x
 \le C\norm u_{L^2(E)}^2+\frac CR\norm g_2^2.
 \label{eq:turning-negative}
\end{equation}
The estimate is also valid if there are no selected intervals.

\step{Step 2. The region \(q_\lambda\ge2R\).}
Set
\[
 \eta_R(x)=
 \begin{cases}
 0,&q_\lambda(x)\le R,\\
 (q_\lambda(x)-R)/R,&R<q_\lambda(x)<2R,\\
 1,&q_\lambda(x)\ge2R.
 \end{cases}
\]
This function is a bounded Lipschitz multiplier, with \(|\eta_R'|\le R^{-1}\). Using \cref{lem:multipliers} and dropping the nonnegative derivative term, we obtain
\[
 R\norm{\eta_Ru}_2^2
 \le \norm g_2\norm{\eta_Ru}_2
 +R^{-2}\int_{\{R<q_\lambda<2R\}}|u|^2\dd x.
\]
Young's inequality gives
\begin{equation}
 \int_{\{q_\lambda\ge2R\}}|u|^2\dd x
 \le R^{-2}\norm g_2^2+
 2R^{-3}\int_{\{|q_\lambda|\le2R\}}|u|^2\dd x.
 \label{eq:turning-positive}
\end{equation}
All terms are justified in the form domain; \(\eta_R\) need not be compactly supported.

\step{Step 3. Combination.}
Add \eqref{eq:turning-negative}, \eqref{eq:turning-positive}, and the mass in \(\{|q_\lambda|\le2R\}\). Since \(R\ge1\), the coefficient of this last mass is at most \(1+2R^{-3}\le3\), and the residual coefficient is bounded by \(C_1/R\). This proves \eqref{eq:turning-resolvent}.
\end{proof}

The choice of \(R\) is independent of \(\lambda\) and of the detailed geometry of \(E\). The last term in \eqref{eq:turning-resolvent} must be retained: it includes states whose large total energy is mainly due to their spatial position. It will be treated by the evolution equation and local compactness, not by an additional thinness assumption.

\section{Unique continuation from a measurable spatial set}\label{sec:unique}
The proof of sufficiency produces a solution that belongs to \(L^2\) only locally in space. This section establishes the unique-continuation statement in precisely that class. We first state the open-set result used in the proof and then give the measurable-set reduction. A direct proof of the resulting one-dimensional proposition is also supplied in \cref{sec:direct-ucp}.

\subsection{The open-set theorem}
\begin{lemma}[Filippas--Laurent--L\'eautaud]\label{lem:open-ucp}
Let \(I,J\subset\R\) be connected open intervals and let \(V\in L^\infty_{\loc}(J)\) be independent of time. If \(O\subset J\) is a nonempty open interval and
\begin{equation}
 W\in L^2_{\loc}(I\times J),\qquad
 (i\partial_t+\partial_x^2-V(x))W=0
       \quad\text{in }\mathcal D'(I\times J),
 \qquad W=0\quad\text{on }I\times O,
 \label{eq:open-ucp}
\end{equation}
then \(W=0\) on \(I\times J\).
\end{lemma}
\begin{proof}[Deduction from the cited theorem]
Apply~\cite[Theorem~1.5]{FLL25}, in its version for local \(L^2\) solutions, on the connected one-dimensional manifold \(J\). In its notation the metric is \(g=1\), the first-order coefficient is zero, and the zeroth-order coefficient is \(-V\). The metric is Lipschitz, the divergence of the first-order coefficient is zero, and the zeroth-order coefficient is independent of time. In particular all its positive-order time derivatives vanish, so it satisfies the stated Gevrey-in-time hypothesis with values in \(L^\infty\) on each compact spatial interval. A translation and restriction of time gives the conclusion on every relatively compact time interval in \(I\). Exhaustion proves the assertion. The reference does not require a global \(L^2(J)\) norm or a bound for \(V\) on all of \(J\).
\end{proof}

\subsection{Time regularization and spatial traces}
\begin{lemma}\label{lem:ucp-regularization}
Let \(I\subset\R\) be open, \(V\in L^\infty_{\loc}(\R)\) be independent of time, and let
\[
 U\in L^2_{\loc}(I\times\R),\qquad
 (i\partial_t+\partial_x^2-V)U=0\quad\text{in }\mathcal D'(I\times\R).
\]
Choose \(\eta\in C_c^\infty(-1,1)\), and put \(\eta_\varepsilon(t)=\varepsilon^{-1}\eta(t/\varepsilon)\). On
\(I_\varepsilon=\{t\in I:\operatorname{dist}(t,\R\setminus I)>\varepsilon\}\), define
\[
 Y_\varepsilon(t,x)=\int_{-\varepsilon}^{\varepsilon}
           \eta_\varepsilon(s)U(t-s,x)\dd s.
\]
Then
\begin{equation}
 (i\partial_t+\partial_x^2-V)Y_\varepsilon=0
       \quad\text{in }\mathcal D'(I_\varepsilon\times\R),
 \qquad
 Y_\varepsilon\in C^\infty(I_\varepsilon;H^2_{\loc}(\R)).
 \label{eq:ucp-regularized}
\end{equation}
\end{lemma}
\begin{proof}
Fix a compact time interval \(K\Subset I_\varepsilon\) and bounded spatial intervals \(J\Subset J_*\). Let \(K_*\Subset I\) contain all \(t-s\) with \(t\in K\) and \(|s|\le\varepsilon\). On this fixed set the convolution is an ordinary convolution of an \(L^2(K_*;L^2(J_*))\) function with a smooth kernel. For every \(k\ge0\),
\begin{equation}
 \sup_{t\in K}\norm{\partial_t^kY_\varepsilon(t)}_{L^2(J_*)}
 \le\norm{\eta_\varepsilon^{(k)}}_{L^2(\R)}
      \norm U_{L^2(K_*\times J_*)}.
 \label{eq:ucp-time-derivative}
\end{equation}
Differentiability and continuity with values in \(L^2(J_*)\) follow by differentiating the smooth kernel in that norm. Since \(V\) is independent of time, testing the original equation and convolving in time proves the first identity in \eqref{eq:ucp-regularized}. Its time derivatives satisfy
\begin{equation}
 \partial_x^2\partial_t^kY_\varepsilon
  =V\partial_t^kY_\varepsilon-i\partial_t^{k+1}Y_\varepsilon.
 \label{eq:ucp-spatial-derivative}
\end{equation}
For almost every time the right-hand side belongs to \(L^2(J_*)\). \Cref{lem:local-elliptic} therefore gives local \(H^2\) regularity. It also gives continuity in that norm: apply the estimate to
\(\partial_t^kY_\varepsilon(t+h)-\partial_t^kY_\varepsilon(t)\), and use the \(L^2(J_*)\) continuity of that difference and of its second derivative in \eqref{eq:ucp-spatial-derivative}. The same estimate applied to difference quotients minus \(\partial_t^{k+1}Y_\varepsilon\) proves differentiability in \(H^2(J)\). Iterating gives \eqref{eq:ucp-regularized}. This argument is local in time and does not require integrability up to the endpoints of \(I\).
\end{proof}

\subsection{Unique continuation from positive spatial measure}
\begin{proposition}\label{prop:measurable-ucp}
Let \(I\subset\R\) be a connected open interval and let \(V\in L^\infty_{\loc}(\R)\) be independent of time. Suppose that
\begin{equation}
 U\in L^2_{\loc}(I\times\R),\qquad
 i\partial_tU=-\partial_x^2U+V(x)U
        \quad\text{in }\mathcal D'(I\times\R).
 \label{eq:measurable-ucp-equation}
\end{equation}
If \(S\subset\R\) is measurable, \(|S|>0\), and \(U=0\) almost everywhere on \(I\times S\), then \(U=0\) almost everywhere on \(I\times\R\).
\end{proposition}
\begin{proof}
It suffices to prove the assertion on an arbitrary bounded subinterval of \(I\). We consequently write \(I=(a,b)\), with finite endpoints; restriction and exhaustion will give the general case. Choose a bounded interval meeting \(S\) in positive measure and replace \(S\) by that intersection. Thus \(0<|S|<\infty\).

\step{Step 1. Regularization and a common spatial zero set.}
Choose \(\eta\ge0\) as in \cref{lem:ucp-regularization}, with \(\int\eta=1\), and take \(0<\varepsilon<(b-a)/4\). Set
\[
 I_\varepsilon=(a+\varepsilon,b-\varepsilon),\qquad
 Y_\varepsilon=\eta_\varepsilon*_tU.
\]
By Fubini's theorem and the vanishing assumption, \(Y_\varepsilon=0\) almost everywhere on \(I_\varepsilon\times S\). There is consequently a full-measure set \(D\subset I_\varepsilon\) such that, for every \(t\in D\), the continuous spatial representative of \(Y_\varepsilon(t)\) is zero almost everywhere on \(S\). Choose a countable dense subset \(\{t_k:k\ge1\}\) of \(D\). For each \(k\), remove the null subset of \(S\) on which that vanishing fails, and call the remaining intersection \(S_\varepsilon\). Then \(|S_\varepsilon|=|S|>0\) and
\[
 Y_\varepsilon(t_k,x)=0\qquad(k\ge1,\ x\in S_\varepsilon).
\]
By \cref{lem:ucp-regularization} and the one-dimensional embedding \(H^2(J)\hookrightarrow C^1(\overline J)\), evaluation at a fixed spatial point is continuous in time. Hence
\begin{equation}
 Y_\varepsilon(t,x)=0
       \qquad(t\in I_\varepsilon,\ x\in S_\varepsilon).
 \label{eq:ucp-common-zero-set}
\end{equation}
The representatives used on overlapping intervals agree, since two continuous functions that agree almost everywhere agree everywhere.

\step{Step 2. A fixed point with zero Cauchy data.}
Choose a Lebesgue density point \(x_0\) of \(S_\varepsilon\), and distinct points \(x_k\in S_\varepsilon\) tending to \(x_0\). The spatial \(C^1\) regularity gives, for each \(t\in I_\varepsilon\),
\begin{equation}
 Y_\varepsilon(t,x_0)=0,
 \qquad
 \partial_xY_\varepsilon(t,x_0)
 =\lim_{k\to\infty}
     \frac{Y_\varepsilon(t,x_k)-Y_\varepsilon(t,x_0)}{x_k-x_0}=0.
 \label{eq:ucp-zero-traces}
\end{equation}
The same \(x_0\) works for every time. It is allowed to depend on \(\varepsilon\).

\step{Step 3. Extension by zero across the spatial point.}
Define
\[
 W_\varepsilon^+(t,x)=\ind_{(x_0,\infty)}(x)Y_\varepsilon(t,x).
\]
For a spatial test function \(\varphi\), two integrations by parts on \((x_0,\infty)\) give
\[
 \int_{x_0}^\infty Y_\varepsilon\varphi''\dd x
 =\int_{x_0}^\infty\partial_x^2Y_\varepsilon\varphi\dd x
   +\partial_xY_\varepsilon(t,x_0)\varphi(x_0)
   -Y_\varepsilon(t,x_0)\varphi'(x_0).
\]
This identity is valid for \(H^2\) functions by approximation and their trace continuity. Integrating also against a compactly supported time test function proves the distributional formula
\[
 \partial_x^2W_\varepsilon^+
 =\ind_{(x_0,\infty)}\partial_x^2Y_\varepsilon
   +\partial_xY_\varepsilon(t,x_0)\delta_{x_0}
   +Y_\varepsilon(t,x_0)\delta'_{x_0}.
\]
Both terms supported at \(x_0\) vanish by \eqref{eq:ucp-zero-traces}. The indicator is independent of time, so it commutes with \(\partial_t\); it also commutes with multiplication by \(V\). Therefore
\begin{equation}
 (i\partial_t+\partial_x^2-V)W_\varepsilon^+=0
       \quad\text{in }\mathcal D'(I_\varepsilon\times\R).
 \label{eq:ucp-zero-extension}
\end{equation}
This function is locally \(L^2\) and vanishes on the open set \(I_\varepsilon\times(-\infty,x_0)\). For each integer \(N>|x_0|+1\), apply \cref{lem:open-ucp} on the connected interval \((-N,N)\), with any nonempty open interval to the left of \(x_0\) as the zero set. It follows that \(W_\varepsilon^+=0\) on \(I_\varepsilon\times(-N,N)\). Letting \(N\to\infty\) gives its vanishing on the whole line.

For the left extension \(W_\varepsilon^-=
\ind_{(-\infty,x_0)}Y_\varepsilon\), the boundary distributions have the opposite signs and vanish for the same reason. It is a homogeneous solution vanishing to the right of \(x_0\), so \cref{lem:open-ucp} gives \(W_\varepsilon^-=0\). The two conclusions prove \(Y_\varepsilon=0\) on \(I_\varepsilon\times\R\).

\step{Step 4. Limit of the time regularization.}
Fix a compact interval \(K\Subset I\) and a bounded spatial interval \(J\). Choose a smooth time cutoff equal to one on a neighborhood of \(K\) and supported compactly in \(I\). After multiplication by this cutoff, \(U\) belongs to \(L^2(\R_t\times J)\), and ordinary time mollification converges to it in that space. On \(K\), this mollification agrees with \(Y_\varepsilon\) for all sufficiently small \(\varepsilon\). Thus \(Y_\varepsilon\to U\) in \(L^2(K\times J)\). Since \(Y_\varepsilon=0\), we obtain \(U=0\) there. Exhausting \(I\) and \(\R\) proves the proposition.
\end{proof}

\begin{remark}\label{rem:ucp-scope}
The proposition is an assertion about time-dependent solutions, not just eigenfunctions. It applies in particular to \(V(y)=|m+y|-|m|\), \(V(y)=y\), and \(V(y)=-y\). No global square-integrability or initial-value representation for the solution is required. In \cref{sec:arbitrary-time}, it is applied to the weak limit of the original translated solutions. The function obtained after the temporal Fourier multiplier has a source term and is not used as a homogeneous solution.
\end{remark}

\section{Observability in an arbitrary positive time}\label{sec:arbitrary-time}
We prove the uniform estimate \eqref{eq:uniform-anytime}. The following selection lemma will be used to retain both nonzero mass and small observation after a spatial translation.

\subsection{A selection lemma}
\begin{lemma}\label{lem:weighted-selection}
For each \(j\ge1\), let \((A_{j,m})_{m\in\Z}\), \((B_{j,m})_{m\in\Z}\), and \((D_{j,m})_{m\in\Z}\) be nonnegative families. Suppose that constants \(M_0,M_1,b_0>0\), independent of \(j\), satisfy
\begin{equation}
 \sum_mA_{j,m}\le M_0,\qquad
 B_{j,m}\le M_1A_{j,m},\qquad
 \sum_mB_{j,m}\ge b_0,
 \qquad
 \sum_mD_{j,m}\le d_j,\quad d_j\longrightarrow0.
 \label{eq:selection-assumptions}
\end{equation}
There are \(a_0>0\), numbers \(\varepsilon_j\to0\), and, after omission of finitely many indices, integers \(m_j\) such that
\begin{equation}
 A_{j,m_j}>0,\qquad
 B_{j,m_j}\ge a_0A_{j,m_j},\qquad
 D_{j,m_j}\le\varepsilon_jA_{j,m_j}.
 \label{eq:selection-result}
\end{equation}
One may take \(a_0=b_0/(8M_0)\) and \(\varepsilon_j=\sqrt{d_j}+j^{-1}\).
\end{lemma}
\begin{proof}
For fixed \(j\), divide the integer indices into the two exceptional sets
\[
 \mathcal B_j^{(1)}=\{m:B_{j,m}<a_0A_{j,m}\},\qquad
 \mathcal B_j^{(2)}=\{m:D_{j,m}>\varepsilon_jA_{j,m}\}.
\]
The first set satisfies
\[
 \sum_{m\in\mathcal B_j^{(1)}}B_{j,m}
 \le a_0\sum_mA_{j,m}\le b_0/8.
\]
For the second set, use both inequalities involving \(A_{j,m}\) to obtain
\[
 \sum_{m\in\mathcal B_j^{(2)}}B_{j,m}
 \le\frac{M_1}{\varepsilon_j}\sum_mD_{j,m}
 \le\frac{M_1d_j}{\sqrt{d_j}+j^{-1}}
 \le M_1\sqrt{d_j}\longrightarrow0.
\]
Hence, for all sufficiently large \(j\), the sum of \(B_{j,m}\) over the complement of their union is at least \(b_0/2\). Choose an index in that complement at which \(B_{j,m}>0\). It satisfies the last two inequalities in \eqref{eq:selection-result}, and \(B_{j,m}\le M_1A_{j,m}\) implies \(A_{j,m}>0\). This proves the lemma.
\end{proof}

\subsection{Proof of the uniform observability estimate}
\begin{proof}[Proof of the sufficient assertion in \cref{thm:main}]
Fix \(T>0\), \(L>0\), and \(0<\gamma\le1\). If \(\gamma=1\), the complement of \(E\) has measure zero on each interval \([kL,(k+1)L]\), and hence on the entire line. Unitarity then gives \(\Cobs(E,T)=1/T\). We may therefore suppose \(0<\gamma<1\).

\step{Step 1. The contradiction sequence and its regularity.}
Let \(\mathcal E_{L,\gamma}\) be the class of all \((L,\gamma)\)-thick sets. Failure of \eqref{eq:uniform-anytime} means that
\[
 \inf_{E\in\mathcal E_{L,\gamma}}\
 \inf_{\norm f_2=1}
 \int_0^T\norm{\e^{-itH}f}_{L^2(E)}^2\dd t=0.
\]
Choose measurable sets \(E_j\in\mathcal E_{L,\gamma}\) and unit vectors whose observation integrals tend to zero. These vectors may be approximated in \(L^2\) by elements of \(D(H)\), and the approximations may be normalized. Indeed, if \(h\) approximates a unit vector \(f\) and \(\norm{h-f}_2<1/2\), then
\[
 \left\|\frac h{\norm h_2}-f\right\|_2\le2\norm{h-f}_2,
 \qquad
 \norm{\ind_E\e^{-itH}(h-f)}_{L^2((0,T)\times\R)}
 \le\sqrt T\norm{h-f}_2.
\]
Thus we obtain
\begin{equation}
 \begin{split}
 &f_j\in D(H),\quad \norm{f_j}_2=1,\quad
 u_j(t)=\e^{-itH}f_j,\\
 &\mathcal O_j:=\int_0^T\norm{u_j(t)}_{L^2(E_j)}^2\dd t
       \longrightarrow0.
 \end{split}
 \label{eq:anytime-contradiction}
\end{equation}
There is no uniform bound on \(\norm{Hf_j}_2\), and none will be required.

Choose \(0<t_-<t_+<T\), set \(I_0=(t_-,t_+)\), and fix a real function \(\chi\in C_c^\infty(0,T)\) such that
\begin{equation}
 0\le\chi\le1,\qquad \chi(t)>0\ (t\in I_0),\qquad
 \chi(t)=0\ (t\notin I_0).
 \label{eq:anytime-cutoff}
\end{equation}
Its support is \([t_-,t_+]\). Put \(v_j=\chi u_j\), extended by zero in time. By \cref{lem:graph-fourier}, \(\widehat v_j(\tau)\in D(H)\) for almost every \(\tau\) and
\begin{equation}
 (H+\tau)\widehat v_j(\tau)
       =-i\widehat{\chi'u_j}(\tau).
 \label{eq:anytime-frequency-equation}
\end{equation}
All norms in this identity are well defined for each fixed \(j\).

\step{Step 2. Nonzero mass in the retained region.}
Apply \cref{prop:turning} with \(\lambda=-\tau\) to \(\widehat v_j(\tau)\), and integrate over \(\tau\). The observation term is controlled by \(\mathcal O_j\), because \(\ind_{E_j}\) commutes with the time Fourier transform and \(\norm\chi_\infty\le1\). For the residual, \eqref{eq:anytime-frequency-equation}, Plancherel, and unitarity give
\[
 \int_\R\norm{(H+\tau)\widehat v_j(\tau)}_2^2\dd\tau
  =\int_\R|\chi'(t)|^2\norm{u_j(t)}_2^2\dd t
  =\norm{\chi'}_2^2.
\]
Consequently
\begin{equation}
 \norm\chi_2^2
 \le C_1\mathcal O_j+\frac{C_1}{R}\norm{\chi'}_2^2+3M_j,
 \qquad
 M_j=\iint_{\{|\tau+|x||\le2R\}}
                |\widehat v_j(\tau,x)|^2\dd\tau\dd x.
 \label{eq:turning-spacetime}
\end{equation}
The constant \(C_1\) and the lower threshold for \(R\) depend only on \(L,\gamma\). Choose \(R\ge R_*\) so that
\(C_1\norm{\chi'}_2^2/R\le\norm\chi_2^2/4\), and keep this \(R\) fixed. After discarding finitely many indices, also \(C_1\mathcal O_j\le\norm\chi_2^2/4\). It follows that
\begin{equation}
 M_j\ge c_0:=\frac16\norm\chi_2^2>0.
 \label{eq:positive-turning-mass}
\end{equation}
The order of these choices is important: \(R\) does not depend on \(j\), on the set \(E_j\), or on its spatial location.

\step{Step 3. Temporal frequency cutoffs on unit spatial intervals.}
Let
\[
 Q_m=(m-1/2,m+1/2),\qquad m\in\Z,
\]
and choose a real \(\psi\in C_c^\infty(\R)\) with
\begin{equation}
 0\le\psi\le1,\qquad
 \psi=1\quad\text{on }[-2R-1,2R+1].
 \label{eq:psi-plateau}
\end{equation}
Define
\begin{equation}
 z_{j,m}=\psi(D_t)(\chi(t)\e^{it|m|}u_j(t,x)),
 \qquad
 B_{j,m}=\norm{z_{j,m}}_{L^2(\R_t\times Q_m)}^2.
 \label{eq:local-temporal-functions}
\end{equation}
Multiplication by \(\e^{it|m|}\) has no spatial derivatives. The Fourier shift formula gives
\[
 \mathcal F_t(\chi\e^{it|m|}u_j)(\tau,x)
       =\widehat v_j(\tau-|m|,x).
\]
With \(\sigma=\tau-|m|\), Plancherel on \(Q_m\) therefore yields
\begin{equation}
 B_{j,m}=\int_{Q_m}\int_\R
       |\psi(\sigma+|m|)|^2|\widehat v_j(\sigma,x)|^2\dd\sigma\dd x.
 \label{eq:local-temporal-mass}
\end{equation}
For \(x\in Q_m\), \(\bigl||x|-|m|\bigr|\le1/2\). Hence \(|\sigma+|x||\le2R\) implies \(|\sigma+|m||\le2R+1/2\), where \(\psi=1\). The intervals \(Q_m\) partition \(\R\) up to endpoints, so Tonelli and \eqref{eq:positive-turning-mass} give
\begin{equation}
 \sum_{m\in\Z}B_{j,m}\ge M_j\ge c_0.
 \label{eq:sum-frequency-mass}
\end{equation}

\step{Step 4. A spatial translation and a normalization.}
Set
\begin{equation}
 \rho(y)=(1+|y|)^{-4},\qquad
 C_\rho=\sup_{x\in\R}\sum_{m\in\Z}\rho(x-m)<\infty.
 \label{eq:weight}
\end{equation}
To verify finiteness, choose an integer \(m_0\) with \(|x-m_0|\le1/2\). Then
\[
 \sum_m\rho(x-m)
 \le1+2\sum_{k=1}^\infty(k+1/2)^{-4},
\]
independently of \(x\). Define
\begin{align}
 A_{j,m}&=\int_0^T\int_\R\rho(x-m)|u_j(t,x)|^2\dd x\dd t,\notag\\
 D_{j,m}&=\int_0^T\int_\R
             \rho(x-m)\ind_{E_j}(x)|u_j(t,x)|^2\dd x\dd t.
 \label{eq:weighted-masses}
\end{align}
Tonelli's theorem and \(\norm{u_j(t)}_2=1\) show that
\begin{equation}
 \sum_mA_{j,m}\le C_\rho T,
 \qquad \sum_mD_{j,m}\le C_\rho\mathcal O_j.
 \label{eq:weighted-sums}
\end{equation}
The multiplier in \eqref{eq:local-temporal-functions} is a contraction on \(L^2\) in time. Thus
\[
 B_{j,m}\le\int_0^T\int_{Q_m}|\chi(t)u_j(t,x)|^2\dd x\dd t
 \le C_2A_{j,m},\qquad C_2=(3/2)^4,
\]
since \(\rho(x-m)\ge(3/2)^{-4}\) on \(Q_m\).

Apply \cref{lem:weighted-selection} with \(M_0=C_\rho T\), \(M_1=C_2\), \(b_0=c_0\), and \(d_j=C_\rho\mathcal O_j\). It provides integers \(m_j\), a fixed \(a_0>0\), and \(\varepsilon_j\to0\) for which
\begin{equation}
 B_{j,m_j}\ge a_0A_{j,m_j}>0,
 \qquad D_{j,m_j}\le\varepsilon_jA_{j,m_j}.
 \label{eq:selected-masses}
\end{equation}
Put
\begin{equation}
 \begin{split}
 c_j&=A_{j,m_j}^{1/2}>0,\\
 U_j(t,y)&=c_j^{-1}\e^{it|m_j|}u_j(t,m_j+y),
 \qquad E'_j=E_j-m_j.
 \end{split}
 \label{eq:normalized-translates}
\end{equation}
A change of variables gives exactly
\begin{equation}
 \int_0^T\int_\R\rho(y)|U_j|^2\dd y\dd t=1,
 \qquad
 \int_0^T\int_\R\rho(y)\ind_{E'_j}(y)|U_j|^2\dd y\dd t
      \le\varepsilon_j.
 \label{eq:normalization-identities}
\end{equation}
In particular, for every bounded interval \(J\), there is \(C_J<\infty\) with
\begin{equation}
 \norm{U_j}_{L^2((0,T)\times J)}^2\le C_J,
 \qquad
 \int_0^T\int_{E'_j\cap J}|U_j|^2\dd y\dd t\le C_J\varepsilon_j.
 \label{eq:normalized-local-bounds}
\end{equation}
These inequalities do not require \(c_j\) to be bounded below. They are the only normalization bounds used below.

\step{Step 5. The weak limit and its equation.}
By weak compactness in the single Hilbert space appearing in \eqref{eq:normalization-identities}, pass to a subsequence such that
\begin{equation}
 U_j\rightharpoonup U
  \quad\text{in }\mathcal H_\rho
  :=L^2((0,T)\times\R,\rho(y)\dd t\dd y).
 \label{eq:weighted-weak-limit}
\end{equation}
Restriction to \((0,T)\times J\) is a bounded operator from \(\mathcal H_\rho\) to ordinary \(L^2\) whenever \(J\) is bounded. Thus the weak convergence holds on the whole time interval over every such \(J\).

The translated equation is
\begin{equation}
 i\partial_tU_j=-\partial_y^2U_j+V_j(y)U_j,
 \qquad V_j(y)=|m_j+y|-|m_j|,
 \qquad |V_j(y)|\le|y|.
 \label{eq:translated-potential}
\end{equation}
An integer sequence has a subsequence that is constant, tends to \(+\infty\), or tends to \(-\infty\). Accordingly, after a further extraction,
\begin{equation}
 V_j\longrightarrow V_\infty\quad\text{in }L^\infty_{\loc}(\R),
 \qquad
 V_\infty(y)=
 \begin{cases}
 |m_*+y|-|m_*|,&m_j=m_*,\\
 y,&m_j\to+\infty,\\
 -y,&m_j\to-\infty.
 \end{cases}
 \label{eq:limiting-potentials}
\end{equation}
In each unbounded case the indicated equality holds on a fixed compact interval for all sufficiently large \(j\).

For \(\Phi\in C_c^\infty((0,T)\times\R)\), the distributional equation reads
\[
 \iint U_j(-i\partial_t\Phi+\partial_y^2\Phi-V_j\Phi)\dd t\dd y=0.
\]
On the support of \(\Phi\), the difference \((V_j-V_\infty)\Phi\) tends to zero in \(L^2\), while \(U_j\) is uniformly bounded there by \eqref{eq:normalized-local-bounds}. The remaining terms pass to the limit by \eqref{eq:weighted-weak-limit}. Consequently
\begin{equation}
 i\partial_tU=-\partial_y^2U+V_\infty(y)U
       \quad\text{in }\mathcal D'((0,T)\times\R).
 \label{eq:limit-equation}
\end{equation}
Only a local solution is asserted; no global \(L^2\) initial-value representation for \(U\) is used.

\step{Step 6. Strong compactness of the functions after the Fourier multiplier.}
On all real times let
\begin{equation}
 Z_j=\psi(D_t)(\chi U_j),\qquad
 W_j=\psi(D_t)(\chi' U_j).
 \label{eq:filtered-normalized}
\end{equation}
Their inputs are extended by zero outside \((0,T)\). \Cref{lem:temporal-compactness} applies by \eqref{eq:normalization-identities} and \eqref{eq:translated-potential}. In particular,
\begin{equation}
 i\partial_t Z_j=-\partial_y^2 Z_j+V_jZ_j+iW_j,
 \label{eq:filtered-forced-equation}
\end{equation}
and, for every bounded \(J\),
\begin{equation}
 \sup_j\left(
 \norm{Z_j}_{L^2(\R_t;H^2(J))}
 +\norm{\partial_tZ_j}_{L^2(\R\times J)}
 +\norm{W_j}_{L^2(\R\times J)}\right)<\infty.
 \label{eq:filtered-regularity}
\end{equation}
After a diagonal extraction,
\begin{equation}
 Z_j\longrightarrow Z\quad\text{strongly in }L^2_{\loc}(\R^2),
 \qquad Z=\psi(D_t)(\chi U).
 \label{eq:filtered-limit}
\end{equation}
The identity for the limit follows from the bounded map \(\mathcal A_J\) in the proof of \cref{lem:temporal-compactness}. Thus it uses the weak convergence on the full finite time interval in \eqref{eq:weighted-weak-limit}, not an interchange of a nonlocal convolution with a merely interior limit.

\step{Step 7. The limit is not zero.}
From \eqref{eq:local-temporal-functions} and \eqref{eq:normalized-translates},
\[
 Z_j(t,y)=c_j^{-1}z_{j,m_j}(t,m_j+y).
\]
Hence \eqref{eq:selected-masses} yields
\begin{equation}
 \int_\R\int_{Q_0}|Z_j(t,y)|^2\dd y\dd t
 =\frac{B_{j,m_j}}{A_{j,m_j}}\ge a_0.
 \label{eq:normalized-frequency-mass}
\end{equation}
The local strong convergence alone would not suffice to pass this lower bound to the limit, because the time integral is over \(\R\). The tail estimate \eqref{eq:abstract-time-tail} gives an \(S>0\), independent of \(j\), such that
\[
 \int_{|t|>S}\int_{Q_0}|Z_j|^2\dd y\dd t\le a_0/2.
\]
Strong convergence on a rectangle containing \([-S,S]\times\overline{Q_0}\) now shows that
\begin{equation}
 \int_{-S}^S\int_{Q_0}|Z(t,y)|^2\dd y\dd t\ge a_0/2.
 \label{eq:nonzero-limit}
\end{equation}
In particular \(Z\ne0\).

\step{Step 8. The limiting observation weight.}
Every set \(E'_j\) is still \((L,\gamma)\)-thick. Since \(L^1(\R)\) is separable, bounded sequences in \(L^\infty(\R)\) have weak-star convergent subsequences. We may therefore assume
\begin{equation}
 \ind_{E'_j}\rightharpoonup^*\zeta
        \quad\text{in }L^\infty(\R),\qquad 0\le\zeta\le1.
 \label{eq:limit-indicators}
\end{equation}
The inequalities on \(\zeta\) follow by testing against nonnegative \(L^1\) functions. For any real \(x\), testing against \(\ind_{[x,x+L]}\) gives
\begin{equation}
 \int_x^{x+L}\zeta(y)\dd y\ge\gamma L.
 \label{eq:limit-thickness}
\end{equation}
There is no need to choose a new subsequence for different values of \(x\): weak-star convergence already holds against every \(L^1\) test function. In particular,
\[
 S_0=\{y\in(0,L):\zeta(y)>0\}
\]
has positive measure, since \(\int_0^L\zeta\ge\gamma L>0\).

For a bounded interval \(J\), multiplication by \(\ind_{E'_j\cap J}\) commutes with \(\psi(D_t)\). The multiplier is a contraction, and \(\norm\chi_\infty\le1\); thus \eqref{eq:normalized-local-bounds} implies
\begin{equation}
 \int_\R\int_{E'_j\cap J}|Z_j|^2\dd y\dd t
 \le\int_0^T\int_{E'_j\cap J}|\chi U_j|^2\dd y\dd t
 \le C_J\varepsilon_j\longrightarrow0.
 \label{eq:filtered-observation-zero}
\end{equation}
To pass this conclusion through the varying sets, let \(K\) be a bounded time interval and put \(Q=K\times J\). Then
\begin{align*}
 &\int_Q\ind_{E'_j}(y)|Z_j(t,y)|^2\dd t\dd y
       -\int_Q\zeta(y)|Z(t,y)|^2\dd t\dd y\\
 &\quad=\int_Q\ind_{E'_j}(y)(|Z_j|^2-|Z|^2)\dd t\dd y\\
 &\qquad\quad+\int_J(\ind_{E'_j}-\zeta)(y)
                   \left(\int_K|Z(t,y)|^2\dd t\right)\dd y.
\end{align*}
The first term tends to zero because
\[
 \norm{|Z_j|^2-|Z|^2}_{L^1(Q)}
 \le(\norm{Z_j}_{L^2(Q)}+\norm Z_{L^2(Q)})
            \norm{Z_j-Z}_{L^2(Q)}\longrightarrow0.
\]
The second tends to zero by \eqref{eq:limit-indicators}, as its fixed test function belongs to \(L^1(J)\). Using \eqref{eq:filtered-observation-zero} and nonnegativity, and then exhausting the plane, we obtain
\begin{equation}
 \zeta(y)|Z(t,y)|^2=0\quad\text{for almost every }(t,y)\in\R^2.
 \label{eq:zero-weighted-limit}
\end{equation}
In particular, \(Z(\cdot,y)=0\) for almost every \(y\in S_0\).

\step{Step 9. Removal of the temporal multiplier on the zero set.}
By \eqref{eq:weighted-weak-limit} and Fubini's theorem, for almost every \(y\in(0,L)\) the function
\[
 g_y(t)=\chi(t)U(t,y),
\]
extended by zero, belongs to \(L^2(\R_t)\) and is supported in \([t_-,t_+]\). The identity \eqref{eq:filtered-limit} holds in \(L^2(\R_t\times J)\) on every bounded \(J\). Taking spatial slices outside a null set therefore gives
\[
 Z(\cdot,y)=\psi(D_t)g_y.
\]
For almost every \(y\in S_0\) this function is zero by \eqref{eq:zero-weighted-limit}. \Cref{lem:compact-time-injectivity} applies, since \(\psi=1\) on the nonempty open interval in \eqref{eq:psi-plateau}. It follows that \(g_y=0\). As \(\chi(t)>0\) for every \(t\in I_0\), we conclude
\begin{equation}
 U=0\quad\text{almost everywhere on }I_0\times S_0.
 \label{eq:unfiltered-zero-set}
\end{equation}
Division by \(\chi\) is used here only to infer an almost-everywhere zero; no uniform lower bound for \(\chi\) near its endpoints is needed.

\step{Step 10. Unique continuation and the contradiction.}
The potential \(V_\infty\) in \eqref{eq:limiting-potentials} is independent of time and belongs to \(L^\infty_{\loc}(\R)\). The function \(U\) belongs to \(L^2_{\loc}\), satisfies the homogeneous equation \eqref{eq:limit-equation}, and vanishes on \(I_0\times S_0\), where \(|S_0|>0\). \Cref{prop:measurable-ucp} therefore gives
\[
 U=0\quad\text{on }I_0\times\R.
\]
Since \(\chi=0\) outside \(I_0\), this implies \(\chi U=0\) on the whole time line after extension by zero. Equation \eqref{eq:filtered-limit} then gives \(Z=0\), contrary to \eqref{eq:nonzero-limit}.

This contradiction proves that the infimum in Step~1 is positive. Its reciprocal is an admissible \(C(T,L,\gamma)\). All parameters fixed before the sequence was translated, including \(\chi,R,\psi,\rho,a_0\), depended only on \(T,L,\gamma\). The estimate is consequently uniform over \(\mathcal E_{L,\gamma}\), as claimed. For a fixed thick set this proves implication (i)\(\Rightarrow\)(ii) of \cref{thm:main}.
\end{proof}

\begin{remark}\label{rem:proof-dependencies}
The equation used in the final unique-continuation argument is \eqref{eq:limit-equation}, not \eqref{eq:filtered-forced-equation}. The latter contains the nonzero term arising from \(\chi'\). Local strong convergence is proved for \(Z_j\), while only weak convergence is used for \(U_j\). The identity \eqref{eq:filtered-limit} and the compact-time-support argument in Step~9 relate the two limits. No decay assumption on \(E^c\) occurs in any of these steps.
\end{remark}

\section{Necessity of thickness}\label{sec:necessity}
We now prove that a non-thick set is not observable on any fixed finite time interval. The proof requires an actual sequence whose time observation tends to zero. A failure of a decaying resolvent estimate alone would not suffice.

\subsection{Intervals with small observed measure}
\begin{lemma}\label{lem:holes}
If \(E\subset\R\) is measurable and not thick, then for every integer \(k\ge1\) there is \(z_k\in\R\) such that
\begin{equation}
 |E\cap(z_k-k,z_k+k)|<\frac1k.
 \label{eq:holes}
\end{equation}
If \(|E|>0\), every sequence satisfying \eqref{eq:holes} has \(|z_k|\to\infty\).
\end{lemma}

\begin{proof}
For each fixed interval length \(2k\), non-thickness implies
\[
 \inf_{x\in\R}|E\cap(x,x+2k)|=0.
\]
Indeed, a positive infimum would provide a thickness proportion at that length. This proves \eqref{eq:holes}. Suppose now that \(|E|>0\) and a subsequence of the centers is bounded. Fix a compact interval \([-R,R]\). Along that subsequence, \(k\to\infty\), so \([-R,R]\subset(z_k-k,z_k+k)\) for all sufficiently large \(k\). It follows that \(|E\cap[-R,R]|=0\). As \(R\) is arbitrary, \(|E|=0\), a contradiction.
\end{proof}

Reflection leaves the equation invariant. More precisely, if \((\mathcal Rf)(x)=f(-x)\), then \(\mathcal R\) is unitary and commutes with \(H\), as follows from the form. Hence observability from \(E\) is equivalent to observability from \(-E\). After passing to a subsequence in \cref{lem:holes} and, if needed, reflecting the problem, we may assume that
\begin{equation}
 x_j\longrightarrow+\infty,\qquad R_j\longrightarrow\infty,
 \qquad |E\cap(x_j-R_j,x_j+R_j)|<R_j^{-1}.
 \label{eq:positive-holes}
\end{equation}
Here \(R_j\) denotes the corresponding subsequence of integer radii. No relation between the rates of growth of \(R_j\) and \(x_j\) is required.

\subsection{Translation and the affine limiting equation}
Let \(\mathcal T_x:L^2(\R_y)\to L^2(\R_x)\) be translation,
\[
 (\mathcal T_xf)(z)=f(z-x).
\]
For \(x_j>0\), define the self-adjoint operator
\begin{equation}
 H_j=\mathcal T_{x_j}^*H\mathcal T_{x_j}-x_j.
 \label{eq:Hj}
\end{equation}
In differential notation it is
\[
 H_j=-\frac{d^2}{dy^2}+|x_j+y|-x_j.
\]
It is self-adjoint by unitary conjugation and subtraction of a real constant. On test functions, and on the Gaussian functions used below,
\begin{equation}
 H_j=-\frac{d^2}{dy^2}+y+W_j(y),
 \qquad
 W_j(y)=
 \begin{cases}
 0,&y\ge-x_j,\\
 2(-x_j-y),&y<-x_j.
 \end{cases}
 \label{eq:Wj}
\end{equation}
In particular,
\begin{equation}
 0\le W_j(y)\le2|y|\ind_{\{y<-x_j\}}.
 \label{eq:Wj-bound}
\end{equation}
The unitary group identity is
\begin{equation}
 \e^{-itH}\mathcal T_{x_j}f
 =\e^{-itx_j}\mathcal T_{x_j}\e^{-itH_j}f.
 \label{eq:group-translation}
\end{equation}
The scalar phase on the right does not change any observation norm.

We next solve the comparison equation explicitly. This avoids the need to invoke a convergence theorem for a sequence of unbounded potentials.

\begin{lemma}\label{lem:affine}
Let \(\phi(y)=\pi^{-1/4}\e^{-y^2/2}\), so that \(\norm\phi_2=1\), and put
\begin{equation}
 z(t,y)=\pi^{-1/4}(1+2it)^{-1/2}
              \exp\left(-\frac{y^2}{2(1+2it)}\right),
 \label{eq:free-gaussian}
\end{equation}
where the square root is the continuous branch equal to one at \(t=0\). Then
\begin{equation}
 v(t,y)=\e^{-i(ty+t^3/3)}z(t,y+t^2)
 \label{eq:affine-gaussian}
\end{equation}
satisfies
\begin{equation}
 i\partial_tv=-\partial_y^2v+yv,\qquad v(0)=\phi.
 \label{eq:affine-equation}
\end{equation}
For each fixed \(T>0\), there are \(A_T,c_T>0\) such that
\begin{equation}
 |v(t,y)|^2\le A_T\e^{-c_Ty^2},
 \qquad 0\le t\le T,\quad y\in\R.
 \label{eq:gaussian-tail}
\end{equation}
In addition, \(v\) and its time derivative are continuous in the Schwartz topology on each compact time interval.
\end{lemma}

\begin{proof}
The Fourier transform of the initial Gaussian, or direct differentiation of \eqref{eq:free-gaussian}, gives \(i\partial_tz=-\partial_y^2z\) and \(z(0)=\phi\). Write \(Y=y+t^2\) and \(\Theta=ty+t^3/3\). Then
\[
 \begin{aligned}
 i\partial_tv
 &=\e^{-i\Theta}\bigl((y+t^2)z+i\partial_tz+2it\partial_Yz\bigr),\\
 -\partial_y^2v+yv
 &=\e^{-i\Theta}\bigl((t^2+y)z+2it\partial_Yz-\partial_Y^2z\bigr),
 \end{aligned}
\]
where each expression involving \(z\) is evaluated at \((t,Y)\). These agree by the free equation, proving \eqref{eq:affine-equation}.

Taking the absolute square gives
\begin{equation}
 |v(t,y)|^2
 =\pi^{-1/2}(1+4t^2)^{-1/2}
       \exp\left(-\frac{(y+t^2)^2}{1+4t^2}\right).
 \label{eq:gaussian-exact}
\end{equation}
Since \((y+t^2)^2\ge y^2/2-t^4\), we may take, for example,
\[
 A_T=\pi^{-1/2}\e^{T^4},\qquad c_T=\frac1{2(1+4T^2)}
\]
in \eqref{eq:gaussian-tail}. Each derivative of \eqref{eq:affine-gaussian} is a polynomial in \(y\) times a Gaussian whose real quadratic decay remains positive for \(t\) in a compact interval. This proves the stated Schwartz continuity, including that of the time derivative.
\end{proof}

\subsection{Uniform comparison on a fixed time interval}
\begin{lemma}\label{lem:duhamel-comparison}
With \(H_j\) and \(v\) as above, let
\[
 w_j(t)=\e^{-itH_j}\phi.
\]
For every fixed \(T>0\),
\begin{equation}
 \sup_{0\le t\le T}\norm{w_j(t)-v(t)}_2\longrightarrow0
 \qquad(j\to\infty).
 \label{eq:comparison}
\end{equation}
\end{lemma}

\begin{proof}
For each fixed \(j\) and \(t\), the Schwartz function \(v(t)\) belongs to the form domain of \(\mathcal T_{x_j}^*H\mathcal T_{x_j}\). Its distributional image under \(-\partial_y^2+|x_j+y|\) belongs to \(L^2\). The converse part of \cref{lem:form}, after translation, therefore gives \(v(t)\in D(H_j)\). Moreover \(v\) is \(C^1\) in \(L^2\) and \(H_jv(t)\) is continuous in \(L^2\) on \([0,T]\). The identity
\[
 i\partial_tv=H_jv-W_jv
\]
can consequently be integrated against the unitary group of \(H_j\). In particular,
\begin{equation}
 w_j(t)-v(t)
 =-i\int_0^t\e^{-i(t-s)H_j}W_jv(s)\dd s.
 \label{eq:duhamel-comparison}
\end{equation}
More explicitly, for fixed \(t\in[0,T]\), the function \(s\mapsto\e^{-i(t-s)H_j}v(s)\) is continuously differentiable as an \(L^2\)-valued function on \([0,t]\). The derivative is
\[
 \frac{\mathrm d}{\mathrm ds}\bigl(\e^{-i(t-s)H_j}v(s)\bigr)
 =\e^{-i(t-s)H_j}\bigl(iH_jv(s)+v'(s)\bigr)
 =i\e^{-i(t-s)H_j}W_jv(s).
\]
To justify the product rule, use \(v(s)\in D(H_j)\), differentiability of the unitary group on its domain, and continuity of both \(v'\) and \(H_jv\) in \(L^2\). Integrating from \(0\) to \(t\) gives
\(v(t)-\e^{-itH_j}\phi=i\int_0^t\e^{-i(t-s)H_j}W_jv(s)\dd s\), which is exactly \eqref{eq:duhamel-comparison}. The integral is a Bochner integral: its integrand is continuous in \(L^2\) for every fixed \(j\). No uniform bound on \(W_j\) as a multiplication operator is being used.

Unitarity gives
\begin{equation}
 \sup_{0\le t\le T}\norm{w_j(t)-v(t)}_2
 \le\int_0^T\norm{W_jv(s)}_2\dd s.
 \label{eq:duhamel-norm}
\end{equation}
Using \eqref{eq:Wj-bound} and \eqref{eq:gaussian-tail},
\[
 \norm{W_jv(s)}_2^2
 \le4A_T\int_{y<-x_j}y^2\e^{-c_Ty^2}\dd y
 \quad(0\le s\le T).
\]
The right side tends to zero because \(x_j\to+\infty\), and is independent of \(s\). Thus the right side of \eqref{eq:duhamel-norm} tends to zero. This proves \eqref{eq:comparison}.
\end{proof}

\subsection{Failure of observability on a non-thick set}
\begin{proof}[Necessary part of \cref{thm:main}]
If \(|E|=0\), the observation integral is zero for every initial state and the assertion is immediate. Suppose \(|E|>0\) and \(E\) is not thick. Fix an arbitrary finite \(T>0\). After the reflection and subsequence choices made above, take \(x_j,R_j\) as in \eqref{eq:positive-holes}. Put
\[
 E_j=E-x_j,\qquad f_j=\mathcal T_{x_j}\phi.
\]
Then \(\norm{f_j}_2=1\) and
\begin{equation}
 |E_j\cap(-R_j,R_j)|<R_j^{-1}.
 \label{eq:translated-hole}
\end{equation}
For the comparison solution \(v\), \eqref{eq:gaussian-tail} yields
\begin{align}
 \int_0^T\int_{E_j}|v(t,y)|^2\dd y\dd t
 &\le A_TT\,|E_j\cap(-R_j,R_j)|
       +A_TT\int_{|y|\ge R_j}\e^{-c_Ty^2}\dd y\notag\\
 &\le\frac{A_TT}{R_j}
       +A_TT\int_{|y|\ge R_j}\e^{-c_Ty^2}\dd y
 \longrightarrow0.
 \label{eq:comparison-observation-small}
\end{align}
The restriction operator to \(E_j\) has norm at most one. Hence
\begin{align}
 \int_0^T\norm{w_j(t)}_{L^2(E_j)}^2\dd t
 &\le2\int_0^T\norm{v(t)}_{L^2(E_j)}^2\dd t
       +2T\sup_{0\le t\le T}\norm{w_j(t)-v(t)}_2^2\notag\\
 &\longrightarrow0
 \label{eq:wj-observation-small}
\end{align}
by \cref{lem:duhamel-comparison}. Finally, \eqref{eq:group-translation} gives
\[
 \norm{\e^{-itH}f_j}_{L^2(E)}
 =\norm{w_j(t)}_{L^2(E_j)}.
\]
Combining this with \eqref{eq:wj-observation-small} proves \eqref{eq:main-failure}. If the chosen centers escaped to minus infinity, the identical conclusion for \(E\) follows by applying this construction to \(-E\) and reflecting the resulting states. Since \(T>0\) was arbitrary, no finite observation time is possible for a non-thick set. Together with \cref{sec:arbitrary-time} and the immediate implication (ii)\(\Rightarrow\)(iii), this completes the proof of \cref{thm:main}.
\end{proof}

\begin{remark}
The centers \(x_j\) need not remain in a classically forbidden region, and the proof imposes no relation between \(x_j\) and \(R_j\). The argument compares actual evolutions of fixed-profile initial data. It does not infer non-observability merely from instantaneous support in a gap, and it does not use finite propagation speed.
\end{remark}

\section{Consequences and examples}\label{sec:consequences}

\subsection{An equivalent fixed-width spectral condition}
We record the converse to the spectral consequence of \cref{prop:resolvent-intro}, in order to identify its relation to \cite[Proposition~4.1]{HWW26}. Such criteria are part of the general theory in~\cite{RTTT05}.

\begin{proof}[Proof of \cref{cor:spectral-intro}]
Let \(P_\lambda=\ind_{[\lambda-\varepsilon,\lambda+\varepsilon]}(H)\). If \(u=P_\lambda u\), the spectral theorem gives \(u\in D(H)\) and
\[
 \norm{(H-\lambda)u}_2^2\le\varepsilon^2\norm u_2^2.
\]
Apply \eqref{eq:main-resolvent} and use \(C_0\varepsilon^2=1/2\). Then
\[
 \norm u_2^2\le\tfrac12\norm u_2^2+C_0\norm u_{L^2(E)}^2,
\]
which proves \eqref{eq:cluster-intro}.
\end{proof}

\begin{proposition}\label{prop:four-equivalences}
For a measurable set \(E\subset\R\), each of the following is equivalent to thickness:
\begin{enumerate}
 \item Observability at some positive time, or equivalently at every positive time, for \(\e^{-itH}\).
 \item Existence of \(C<\infty\) such that
 \[
 \norm u_2^2\le C\bigl(\norm{(H-\lambda)u}_2^2+
                                  \norm u_{L^2(E)}^2\bigr)
 \quad(\lambda\in\R,\ u\in D(H)).
 \]
 \item Existence of \(\varepsilon,c>0\) such that for every \(\lambda\in\R\),
 \begin{equation}
  \norm u_{L^2(E)}^2\ge c\norm u_2^2,
  \qquad u\in\Ran\ind_{[\lambda-\varepsilon,\lambda+\varepsilon]}(H).
  \label{eq:abstract-cluster}
 \end{equation}
\end{enumerate}
\end{proposition}

\begin{proof}
The equivalence between thickness and (i) is \cref{thm:main}. Thickness implies (ii) by \cref{prop:resolvent-intro}, and (ii) implies (i) by \cref{prop:abstract-time}. The same abstract proposition also gives (i)\(\Rightarrow\)(ii) directly, with constants depending on an observation time and its observation constant. Condition (ii) implies (iii) by the proof of \cref{cor:spectral-intro}.

For the remaining implication, assume (iii). Fix \(\lambda\in\R\) and \(u\in D(H)\), and decompose
\[
 u=v+w,\qquad v=P_\lambda u,\qquad w=(I-P_\lambda)u,
 \quad P_\lambda=\ind_{[\lambda-\varepsilon,\lambda+\varepsilon]}(H).
\]
The vectors \(v\) and \(w\) are orthogonal in \(L^2(\R)\). Also,
\[
 \norm w_2\le\varepsilon^{-1}\norm{(H-\lambda)u}_2.
\]
By \eqref{eq:abstract-cluster},
\[
 \norm v_2^2
 \le c^{-1}\norm v_{L^2(E)}^2
 \le2c^{-1}\norm u_{L^2(E)}^2+2c^{-1}\norm w_2^2.
\]
Consequently,
\begin{equation}
 \norm u_2^2
 \le2c^{-1}\norm u_{L^2(E)}^2
 +(1+2c^{-1})\varepsilon^{-2}\norm{(H-\lambda)u}_2^2,
 \label{eq:cluster-to-res}
\end{equation}
which is (ii) after enlarging the constant.
\end{proof}

By \cref{lem:compact}, choose an orthonormal eigenbasis \((\phi_j)\) with eigenvalues \((\lambda_j)\), listed with multiplicities. For
\[
 J_\lambda=\{j:|\lambda_j-\lambda|\le\varepsilon\},
\]
condition~\eqref{eq:cluster-intro} means
\begin{equation}
 \int_E\left|\sum_{j\in J_\lambda}c_j\phi_j(x)\right|^2\dd x
 \ge\frac1{2C_0}\sum_{j\in J_\lambda}|c_j|^2.
 \label{eq:gram}
\end{equation}
Thus the Gram matrices have a positive lower bound independent of the center of the spectral interval. The assertion controls all linear combinations in each cluster, not only individual eigenfunctions. The corresponding half-line example in~\cite{HWW26} shows why a statement about individual eigenfunctions is not enough for this operator.

\subsection{Periodic sets, thin complements and finite complements}
The set \(E_{\mathrm{per}}\) in \eqref{eq:periodic-intro} is \((1,1/2)\)-thick, since its indicator is one-periodic and has integral \(1/2\) over a period. It is therefore observable at every positive time by \cref{thm:main}. Its complement is not \(\alpha\)-thin for any positive \(\alpha\). This example, used in \cite[Remark~3.3]{HWW26} to separate the geometric classes, is included without an additional hypothesis at infinity.

\begin{corollary}\label{cor:thin-consequence}
If \(E^c\) satisfies \eqref{eq:thin} for some \(\alpha>0\), then \(E\) is observable at every positive time. The same conclusion holds if
\begin{equation}
 \lim_{R\to\infty}\sup_{|x|\ge R}|E^c\cap[x,x+1]|=0,
 \label{eq:vanishing-complement}
\end{equation}
or if \(|E^c|<\infty\).
\end{corollary}
\begin{proof}
Condition \eqref{eq:thin} implies \eqref{eq:vanishing-complement}. Under \eqref{eq:vanishing-complement}, choose \(R\ge1\) so that every unit interval contained in either \((R,\infty)\) or \((-\infty,-R)\) has observed measure at least \(1/2\). In an arbitrary interval of length \(L=4R+8\), the part outside \([-R,R]\) has length at least \(L-2R\) and at most two components. Each component contains disjoint unit intervals with total length at least its length minus one. Summing their observed measures gives
\[
 |E\cap[x,x+L]|\ge\tfrac12(L-2R-2)\ge L/4.
\]
Thus \(E\) is thick. If \(|F|<\infty\), where \(F=E^c\), choose \(L=2|F|+1\); then
\[
 |E\cap[x,x+L]|\ge L-|F|\ge L/2
\]
for every \(x\). In each case apply \cref{thm:main}.
\end{proof}

The preceding corollary is qualitative. It does not supply an explicit estimate of \(\Cobs(E,T)\) as \(T\downarrow0\) from \(\alpha,M,R_0\), nor does it assert that a positive thinness exponent is necessary.

\begin{example}\label{ex:finite-complement}
Let
\begin{equation}
 F=\bigcup_{j=2}^\infty(\e^j,\e^j+j^{-2}),
 \qquad E=\R\setminus F.
 \label{eq:sparse-finite-complement}
\end{equation}
The intervals are disjoint and \(|F|=\sum_{j\ge2}j^{-2}<\infty\), so \(E\) is observable at every positive time. At \(x_j=\e^j\),
\[
 |F\cap[x_j,x_j+1]|=j^{-2},\qquad
 x_j^\alpha|F\cap[x_j,x_j+1]|=\e^{\alpha j}j^{-2}\longrightarrow\infty
\]
for every \(\alpha>0\). This gives a second direct reason that no positive exponent in \eqref{eq:thin} can be necessary.
\end{example}

Neither half-line is thick, because it misses intervals of arbitrarily large length on the opposite side. The necessary part of \cref{thm:main} therefore recovers the failure of observation on either half-line proved in \cite[Theorem~1.4(ii)]{HWW26}. Every finite-measure observation set also fails thickness: otherwise disjoint intervals of a fixed length would each contain a fixed positive amount of observed measure. The relevant distinction is between a finite-measure observation set and an observation set with finite-measure complement.

\subsection{A decaying global residual coefficient is not necessary}
The arbitrary-time proof retains the turning-region remainder in \eqref{eq:turning-resolvent}. That term cannot in general be discarded in favor of a globally vanishing residual coefficient.

\begin{proposition}\label{prop:no-decay}
For \(E=E_{\mathrm{per}}\) in \eqref{eq:periodic-intro}, there do not exist a function \(a(\lambda)\) tending to zero as \(\lambda\to+\infty\) and a finite constant \(C\) such that
\begin{equation}
 \norm u_2\le a(\lambda)\norm{(H-\lambda)u}_2
                  +C\norm u_{L^2(E)}
 \label{eq:impossible-decay}
\end{equation}
holds for every sufficiently large real \(\lambda\) and every \(u\in D(H)\).
\end{proposition}

\begin{proof}
Choose \(\chi\in C_c^\infty((1/2,1))\) with \(\norm\chi_2=1\). For a positive integer \(j\), put
\[
 u_j(x)=\chi(x-j),\qquad \lambda_j=j.
\]
The function \(u_j\) belongs to \(D(H)\) and is supported in an unobserved interval of \(E\). On its support \(x>0\), and hence
\[
 (H-j)u_j(x)=-\chi''(x-j)+(x-j)\chi(x-j).
\]
It follows that
\[
 \norm{u_j}_{L^2(E)}=0,\qquad\norm{u_j}_2=1,
 \qquad\norm{(H-j)u_j}_2=\norm{-\chi''+y\chi}_{L^2(\R)},
\]
where the last number is fixed and finite. Substitution into \eqref{eq:impossible-decay} would give
\(1\le a(j)\norm{-\chi''+y\chi}_2\to0\), a contradiction.
\end{proof}

The periodic set in \cref{prop:no-decay} is nevertheless observable at every positive time by \cref{thm:main}. Thus the failure of that particular decaying resolvent estimate is not a positive-time obstruction. A function supported in a fixed unobserved interval with residual of fixed size is not a sequence whose evolution has observation tending to zero. The localization in \cref{prop:turning} and the compactness argument in \cref{sec:arbitrary-time} are what permit the general thick-set result.

\subsection{Uniformity and the scope of the conclusion}
For fixed \(T,L,\gamma\), the sufficient proof first chooses a time cutoff \(\chi\), then fixes \(R\) using \eqref{eq:turning-spacetime}, and finally fixes \(\psi\), \(\rho\), and \(a_0\). None of these choices depends on the observation set or on the sequence of initial data. The interval centers and the normalization factors are selected afterwards. The contradiction therefore gives a uniform constant for the full \((L,\gamma)\)-thick class. It does not give its numerical value.

The proof of necessity fixes \(T\) before the centers escape to infinity. Its Gaussian comparison constants may depend on that fixed time. This order of choices proves failure at every individual finite time, without requiring estimates uniform as \(T\to\infty\).

The result is the equivalence in \cref{thm:main} for the fixed potential \(|x|\), together with the dichotomy \eqref{eq:Tmin-dichotomy}. It gives the thick-set answer to \cite[Remark~3.8]{HWW26} in the arbitrary-time sense. It does not assert a corresponding characterization for arbitrary potentials, variable diffusion coefficients, higher dimensions, or time-dependent observation sets. Determining explicit or optimal small-time bounds for general thick sets remains a separate quantitative question.

\appendix
\section{A fixed resolvent estimate and a sufficient observation time}\label{sec:time}
We give the abstract time argument with its constants. Resolvent criteria for conservative systems and their relation to spectral criteria are developed in~\cite{Miller05,RTTT05}; see also the operator-semigroup treatment in~\cite{TW09}.

\begin{proposition}\label{prop:abstract-time}
Let \(A\) be a self-adjoint operator on a complex Hilbert space \(X\), and let \(B:X\to Y\) be bounded, where \(Y\) is another Hilbert space. Suppose there are constants \(M,m>0\) such that
\begin{equation}
 \norm z_X^2\le M\norm{(A-\lambda)z}_X^2+m\norm{Bz}_Y^2,
 \qquad\lambda\in\R,\quad z\in D(A).
 \label{eq:abstract-res}
\end{equation}
Then, for every \(T>\pi\sqrt M\),
\begin{equation}
 \norm f_X^2
 \le \frac{2mT}{T^2-\pi^2M}
        \int_0^T\norm{B\e^{-itA}f}_Y^2\dd t,
 \qquad f\in X.
 \label{eq:abstract-obs}
\end{equation}
Conversely, if for some \(T>0\) one has
\begin{equation}
 \norm f_X^2\le C_T\int_0^T\norm{B\e^{-itA}f}_Y^2\dd t
 \qquad(f\in X),
 \label{eq:assumed-obs}
\end{equation}
then
\begin{equation}
 \norm z_X^2
 \le 2TC_T\norm{Bz}_Y^2+\frac{2T^3C_T\norm B^2}{3}\norm{(A-\lambda)z}_X^2
 \label{eq:converse-res}
\end{equation}
for all \(\lambda\in\R\) and \(z\in D(A)\).
\end{proposition}

\begin{proof}
\step{Step 1. Time cutoff and the Fourier transform.}
We first take \(f\in\Ran P_N\), where \(P_N=\ind_{[-N,N]}(A)\). On this spectral subspace \(A\) is bounded, so every expression below is justified in the Hilbert-space norm, with no domain interchange for an unbounded operator. Put \(u(t)=\e^{-itA}f\) and
\[
 \varphi_T(t)=
 \begin{cases}\sin(\pi t/T),&0<t<T,\\0,&\text{otherwise}.
 \end{cases}
 \qquad w(t)=\varphi_T(t)u(t).
\]
The zero extension \(\varphi_T\) belongs to \(H^1(\R)\). It has no jump at \(0\) or \(T\), and therefore its distributional derivative contains no point masses. We have
\begin{equation}
 (i\partial_t-A)w=i\varphi_T'(t)u(t).
 \label{eq:time-cutoff}
\end{equation}
With the convention in \eqref{eq:fourier}, the time transform gives
\begin{equation}
 (A+\tau)\widehat w(\tau)
 =-i\,\widehat{\varphi_T'u}(\tau).
 \label{eq:time-transform}
\end{equation}
For each \(\tau\), the vector \(\widehat w(\tau)\) belongs to \(\Ran P_N\subset D(A)\). Apply \eqref{eq:abstract-res} with \(\lambda=-\tau\) and integrate over \(\tau\). Since a bounded \(B\) commutes with the Bochner Fourier transform, Hilbert-valued Plancherel yields
\begin{equation}
 \int_{\R}\norm{w(t)}_X^2\dd t
 \le M\int_{\R}|\varphi_T'(t)|^2\norm{u(t)}_X^2\dd t
       +m\int_{\R}|\varphi_T(t)|^2\norm{Bu(t)}_Y^2\dd t.
 \label{eq:time-plancherel}
\end{equation}

\step{Step 2. Absorption and the time threshold.}
The unitary group preserves \(\norm f_X\), and direct integration gives
\[
 \int_0^T|\varphi_T|^2\dd t=\frac T2,
 \qquad \int_0^T|\varphi_T'|^2\dd t=\frac{\pi^2}{2T},
 \qquad |\varphi_T|\le1.
\]
Consequently,
\[
 \frac T2\norm f_X^2
 \le\frac{M\pi^2}{2T}\norm f_X^2
        +m\int_0^T\norm{B\e^{-itA}f}_Y^2\dd t.
\]
If \(T^2>\pi^2M\), moving the first term on the right to the left proves \eqref{eq:abstract-obs} for \(f\in\Ran P_N\).

For arbitrary \(f\in X\), the spectral theorem gives \(P_Nf\to f\) in \(X\). Moreover,
\[
 \int_0^T\norm{B\e^{-itA}(P_Nf-f)}_Y^2\dd t
 \le T\norm B^2\norm{P_Nf-f}_X^2\longrightarrow0.
\]
Thus \eqref{eq:abstract-obs} passes to every \(f\in X\).

\step{Step 3. Converse implication.}
Fix \(z\in D(A)\) and \(\lambda\in\R\). Comparing \(\e^{-itA}z\) with \(\e^{-it\lambda}z\) gives
\begin{equation}
 \e^{-itA}z-\e^{-it\lambda}z
 =-i\int_0^t\e^{-i(t-s)A}\e^{-is\lambda}(A-\lambda)z\dd s.
 \label{eq:abstract-duhamel}
\end{equation}
This follows by differentiating on \(D(A)\), or by integrating the strong evolution equation. Hence
\[
 \norm{B\e^{-itA}z}_Y
 \le\norm{Bz}_Y+t\norm B\norm{(A-\lambda)z}_X.
\]
Square using \((a+b)^2\le2a^2+2b^2\), integrate, and apply \eqref{eq:assumed-obs}. Since \(\int_0^Tt^2\dd t=T^3/3\), the result is \eqref{eq:converse-res}.
\end{proof}

\begin{proof}[Proof of \eqref{eq:main-time}]
Take \(X=L^2(\R)\), \(Y=L^2(E)\) and let \(B\) be restriction to \(E\). Then \(\norm B\le1\). \Cref{prop:resolvent-intro} supplies \eqref{eq:abstract-res} with \(M=m=C_0(L,\gamma)\). Applying \cref{prop:abstract-time} gives \eqref{eq:main-time}. In particular, at least one observation time exists whenever \(E\) is thick.
\end{proof}

\begin{remark}\label{rem:threshold}
The strict inequality \(T>\pi\sqrt{C_0}\) arises from the chosen time cutoff and the fixed residual coefficient. No conclusion at equality, or below this sufficient threshold, follows from the calculation. Also, observability at one time implies observability at every larger time by enlarging the nonnegative time integral. This monotonicity does not reduce the minimal time to zero.
\end{remark}

\section{A direct one-dimensional unique-continuation proof}\label{sec:direct-ucp}
This appendix proves \cref{prop:measurable-ucp} without using the open-set theorem in \cref{lem:open-ucp}. Its assumptions remain exactly those of that proposition: the solution is locally \(L^2\), and the potential is locally bounded and independent of time. The proof uses time regularization of order strictly less than two and a twice-integrated spatial equation. The construction of the regularizing kernel is included.

\subsection{A compactly supported regularizing kernel}
\begin{lemma}\label{lem:gevrey-kernel}
There is a nonnegative function \(\eta\in C_c^\infty(\R)\), with \(\supp\eta=[-1,1]\) and \(\int\eta=1\), for which
\begin{equation}
 \norm{\eta^{(k)}}_{L^2(\R)}
 \le C A^k(k!)^{3/2}\qquad(k\ge0)
 \label{eq:gevrey-kernel-bound}
\end{equation}
with fixed constants \(C,A>0\). Consequently, for \(\eta_\varepsilon(t)=\varepsilon^{-1}\eta(t/\varepsilon)\),
\begin{equation}
 \norm{\eta_\varepsilon^{(k)}}_2
 \le C\varepsilon^{-1/2}(A/\varepsilon)^k(k!)^{3/2}.
 \label{eq:gevrey-scaled-bound}
\end{equation}
\end{lemma}
\begin{proof}
Define
\[
 f(s)=\begin{cases}\exp(-s^{-2}),&s>0,\\0,&s\le0.\end{cases}
\]
For a real \(s>0\), the holomorphic function \(z\mapsto\exp(-z^{-2})\) is defined on the disk \(|z-s|\le s/8\). Write \(z=s(1+w)\), \(|w|\le1/8\). Since
\(
 \operatorname{Re}(1+w)^{-2}\ge1/2,
\)
Cauchy's formula gives
\begin{equation}
 |f^{(k)}(s)|\le k!(8/s)^k\exp(-1/(2s^2)).
 \label{eq:gevrey-cauchy}
\end{equation}
For completeness, the stated bound on the real part follows from
\(|1+w|\le9/8\) and \(|\arg(1+w)|\le\arctan(1/7)\):
\[
 \operatorname{Re}(1+w)^{-2}
 =|1+w|^{-2}\cos(2\arg(1+w))
 \ge(8/9)^2\cos(2\arctan(1/7))>1/2.
\]
For fixed \(k\), the right-hand side of \eqref{eq:gevrey-cauchy} tends to zero as \(s\downarrow0\). Induction on \(k\) therefore shows that the zero extension of \(f\) is smooth and all of its derivatives vanish at zero.

For \(k\ge1\), the maximum of \(s^{-k}\exp(-1/(2s^2))\) over \(s>0\) equals \((k/\mathrm e)^{k/2}\). The elementary estimate \(k!\ge(k/\mathrm e)^k\) implies
\[
 \norm{f^{(k)}}_{L^\infty(\R)}\le8^k(k!)^{3/2}.
\]
For \(k=0\), \(\norm f_\infty\le1\). Set
\[
 \eta(t)=c\,f(1-t)f(1+t),
 \qquad
 c=\left(\int_{-1}^1 f(1-t)f(1+t)\dd t\right)^{-1}.
\]
The normalizing denominator is positive and finite. The product is smooth, vanishes outside \([-1,1]\), and is strictly positive inside. By Leibniz' rule, for \(k\ge1\),
\begin{align*}
 \norm{\eta^{(k)}}_\infty
 &\le c8^k\sum_{\ell=0}^k\binom{k}{\ell}
                    (\ell!)^{3/2}((k-\ell)!)^{3/2}\\
 &=c8^k(k!)^{3/2}
             \sum_{\ell=0}^k\binom{k}{\ell}^{-1/2}
 \le c8^k(k+1)(k!)^{3/2}
 \le c16^k(k!)^{3/2}.
\end{align*}
The case \(k=0\) is absorbed into a fixed constant. As the support has length two, this gives \eqref{eq:gevrey-kernel-bound}. Finally,
\(\norm{\eta_\varepsilon^{(k)}}_2
=\varepsilon^{-k-1/2}\norm{\eta^{(k)}}_2\), proving \eqref{eq:gevrey-scaled-bound}.
\end{proof}

\subsection{The integral equation}
\begin{proof}[Direct proof of \cref{prop:measurable-ucp}]
Let \(U,V,I,S\) satisfy the assumptions of that proposition. As in its proof, it suffices to treat a bounded time interval \(I=(a,b)\) and a bounded positive-measure subset of \(S\). Fix \(0<\varepsilon<(b-a)/4\), use the kernel from \cref{lem:gevrey-kernel}, and put
\[
 Y=\eta_\varepsilon*_tU,
 \qquad I_\varepsilon=(a+\varepsilon,b-\varepsilon).
\]
\Cref{lem:ucp-regularization} gives
\begin{equation}
 i\partial_tY=-\partial_x^2Y+VY,
 \qquad Y\in C^\infty(I_\varepsilon;H^2_{\loc}(\R)).
 \label{eq:direct-ucp-equation}
\end{equation}
The common-zero-set argument in Steps~1--2 of the proof of \cref{prop:measurable-ucp} uses only this regularization and the original zero set. It therefore provides a point \(x_0\in\R\), independent of time, with
\begin{equation}
 Y(t,x_0)=\partial_xY(t,x_0)=0
       \qquad(t\in I_\varepsilon).
 \label{eq:direct-zero-data}
\end{equation}
No open-set unique-continuation theorem has been used up to this point.

\step{Step 1. A derivative bound with constants independent of the order.}
Fix a compact \(K\Subset I_\varepsilon\), a bounded interval \(J\) containing \(x_0\), and a bounded interval \(J_*\) with \(\overline J\subset J_*\). Choose a compact \(K_*\Subset I\) containing the \(\varepsilon\)-neighborhood of \(K\). Equations \eqref{eq:ucp-time-derivative} and \eqref{eq:gevrey-scaled-bound} imply
\[
 \sup_{t\in K}\norm{\partial_t^kY(t)}_{L^2(J_*)}
 \le C_{\varepsilon,K,J_*} A_\varepsilon^k(k!)^{3/2},
\]
where \(A_\varepsilon\ge1\) is fixed for the chosen \(\varepsilon\). Combining this with
\[
 \partial_x^2\partial_t^kY
       =V\partial_t^kY-i\partial_t^{k+1}Y
\]
and \cref{lem:local-elliptic}, followed by \(H^2(J)\hookrightarrow C^1(\overline J)\), gives
\begin{equation}
 \sup_{t\in K}\sup_{x\in J}|\partial_t^kY(t,x)|
 \le M A_\varepsilon^k((k+1)!)^{3/2}
       \qquad(k\ge0).
 \label{eq:direct-all-derivatives}
\end{equation}
Here \(M\) may depend on \(U,\varepsilon,K,J,J_*\) and \(\norm V_{L^\infty(J_*)}\), but it is independent of \(k\). The extra factor \(A_\varepsilon\) from the derivative of order \(k+1\) is absorbed into \(M\). This fixes the constants before the iteration index is introduced.

\step{Step 2. Integration twice in the spatial variable.}
The Sobolev fundamental theorem of calculus and \eqref{eq:direct-zero-data} give
\begin{equation}
 Y(t,x)=\int_{x_0}^x(x-y)
             \bigl(V(y)Y(t,y)-i\partial_tY(t,y)\bigr)\dd y.
 \label{eq:direct-volterra}
\end{equation}
This identity holds for all \(t\in I_\varepsilon\) and every \(x\) in a bounded interval containing \(x_0\), using the continuous representatives. Write
\[
 (\mathcal JF)(x)=\int_{x_0}^x(x-y)F(y)\dd y.
\]
Then \eqref{eq:direct-volterra} says \(Y=(\mathcal JV-i\mathcal J\partial_t)Y\). The operators \(\mathcal J\) and multiplication by \(V\) are independent of time and commute with time differentiation on these regularized functions. They are not assumed to commute with one another.

\step{Step 3. Iteration with its full factorial bound.}
Put
\[
 Q=\norm V_{L^\infty(J)},\qquad
 r=\sup_{x\in J}|x-x_0|.
\]
Iterate \eqref{eq:direct-volterra} \(m\) times. The expansion is a sum of ordered words in \(\mathcal JV\) and \(-i\mathcal J\partial_t\). There are \(\binom mk\) words containing \(k\) time derivatives. For \(x\ge x_0\), the spatial kernel of \(\mathcal J\) is nonnegative. Replacing every factor \(V\) by \(Q\) and moving the time derivatives to the function on the right shows that each such word is bounded in absolute value by
\[
 Q^{m-k}\mathcal J^m|\partial_t^kY(t,\cdot)|(x).
\]
Repeated integration gives
\begin{equation}
 (\mathcal J^mF)(x)
   =\frac1{(2m-1)!}\int_{x_0}^x(x-y)^{2m-1}F(y)\dd y,
 \qquad \mathcal J^m1=\frac{(x-x_0)^{2m}}{(2m)!}.
 \label{eq:iterated-volterra-kernel}
\end{equation}
The formula follows by induction: after changing the order of integration, the integral of the product of the two kernels is the corresponding beta integral, or may be evaluated by two elementary integrations. For \(x<x_0\), set \(s=x_0-x\); the same calculation applies to the reflected variable and gives the bound with \(|x-x_0|\).

It follows from \eqref{eq:direct-all-derivatives} that, for \((t,x)\in K\times J\),
\begin{align}
 |Y(t,x)|
 &\le \frac{M r^{2m}}{(2m)!}
       \sum_{k=0}^m\binom mkQ^{m-k}A_\varepsilon^k
                      ((k+1)!)^{3/2}\notag\\
 &\le M\frac{[r^2(Q+A_\varepsilon)]^m
                     ((m+1)!)^{3/2}}{(2m)!}.
 \label{eq:direct-factorial-bound}
\end{align}
All constants in this inequality were fixed before \(m\). Moreover,
\[
 (2m)!\ge(m!)^2,
 \qquad
 \frac{((m+1)!)^{3/2}}{(2m)!}
 \le\frac{(m+1)^{3/2}}{(m!)^{1/2}}.
\]
Using \(m!\ge(m/\mathrm e)^m\), the right-hand side of \eqref{eq:direct-factorial-bound} is at most
\[
 M(m+1)^{3/2}
     \left(\frac{\mathrm e^{1/2}r^2(Q+A_\varepsilon)}{m^{1/2}}\right)^m,
\]
which tends to zero as \(m\to\infty\). Hence \(Y=0\) on \(K\times J\). Every compact spatial interval is contained in one that also contains \(x_0\), and \(K\Subset I_\varepsilon\) was arbitrary. Thus \(Y=0\) on \(I_\varepsilon\times\R\).

\step{Step 4. Removal of regularization.}
As in Step~4 of the proof of \cref{prop:measurable-ucp}, compactly supported time cutoffs reduce the approximation to the usual \(L^2\) approximate-identity property. Therefore \(\eta_\varepsilon*_tU\to U\) on every compact time-space rectangle as \(\varepsilon\downarrow0\). Since each regularized function vanishes, \(U=0\) on \(I\times\R\). This completes the direct proof.
\end{proof}

\begin{remark}
The iteration uses neither derivatives of \(V\) in the spatial variable nor a global bound on \(V\). Time independence is used when moving time derivatives through the integral operators. The derivative order \(3/2<2\) is used only to make the numerator in \eqref{eq:direct-factorial-bound} smaller than the factorial produced by twice integrating in space. These are the precise features needed for the local \(L^2\) unique-continuation statement in the main proof.
\end{remark}

\end{document}